\documentclass[10pt,reqno]{amsart}
\usepackage{amsmath}
\usepackage{amsfonts}
\usepackage{amssymb}
\usepackage{verbatim}
\usepackage{epsfig}  		
\usepackage{epic,eepic}       
\usepackage{tikz}
\usepackage{tikz-cd}
\usetikzlibrary{matrix, calc, arrows}
\usepackage{hyperref}

\newtheorem{thm}{Theorem}[section]
\newtheorem{cor}[thm]{Corollary}
\newtheorem{lem}[thm]{Lemma}
\newtheorem{mydef}[thm]{Definition}
\newtheorem{rem}[thm]{Remark}
\newtheorem{ex}[thm]{Example}
\newtheorem{prop}[thm]{Proposition}

\newcommand{\ric}[1]{\text{Ric}(#1)}

\newcommand{\xRightarrow}[2][]{\ext@arrow 0359\Rightarrowfill@{#1}{#2}}

\newtheorem{thmx}{Theorem}

\begin{document}

\title[K-semistability of cscK manifolds]{K-semistability of cscK manifolds with transcendental cohomology class}
\author{Zakarias Sj\"ostr\"om Dyrefelt}
\email{zakarias.sjostrom@math.univ-toulouse.fr}
\address{Insitut de Math\'ematiques de Toulouse, Universit\'e Paul Sabatier, 118 route de Narbonne, F-31062 Toulouse Cedex 9 and Centre de Math\'ematiques Laurent Schwartz, \'Ecole Polytechnique, 91 128 Palaiseau Cedex }


\begin{abstract}
We prove that constant scalar curvature K\"ahler (cscK) manifolds with transcendental cohomology class are K-semistable, naturally generalising the situation for polarised manifolds. Relying on a very recent result by R. Berman, T. Darvas and C. Lu regarding properness of the K-energy, it moreover follows that cscK manifolds with finite automorphism group are uniformly K-stable. 
As a main step of the proof we establish, in the general K\"ahler setting, a formula relating the (generalised) Donaldson-Futaki invariant to the asymptotic slope of the K-energy along weak geodesic rays. 
\end{abstract}

\maketitle


\section{Introduction} 

\noindent In this paper we are interested in questions of stability for constant scalar curvature K\"ahler (cscK) manifolds with \emph{transcendental}\footnote{We use the word 'transcendental' to emphasise that the K\"ahler class in question is not necessarily of the form $c_1(L)$ for some ample line bundle $L$ over $X$.} cohomology class. To this end, let  $(X,\omega)$ be a compact K\"ahler manifold and $\alpha := [\omega] \in H^{1,1}(X,\mathbb{R})$ the corresponding K\"ahler class. 
When $\alpha$ is the first Chern class $c_1(L)$ of some ample line bundle $L$ over $X$, such questions are closely related to the Yau-Tian-Donaldson (YTD) conjecture \cite{Donaldsontoric, Tian, Yauconjecture}: \emph{A polarised algebraic manifold $(X,L)$ is K-polystable if and only if the polarisation class $c_1(L)$ admits a K\"ahler metric of constant scalar curvature. }
This conjecture was recently confirmed in the \emph{Fano case}, i.e. when $L = -K_X$, cf. \cite{CDSone, CDStwo, CDSthree, TianYTDconjecture}. In this important special case, a cscK metric is nothing but a K\"ahler-Einstein metric. For general polarised cscK manifolds, the "if" direction of the YTD conjecture was initially proven by Mabuchi in \cite{Mabuchi08}, see also \cite{Berman}. Prior to that, several partial results had been obtained by Donaldson \cite{Donaldsoncalabi} and Stoppa \cite{StoppacscKmetricsimpliesstability}, both assuming that $c_1(L)$ contains a cscK metric. 

For transcendental classes very little is currently known about the validity of a correspondence between \emph{existence of cscK metrics} and \emph{stability} in the spirit of the YTD conjecture. 
Moreover, from a differential geometric point of view, there is no special reason to restrict attention to K\"ahler manifolds with associated integral (or rational) cohomology classes, which are then automatically of the form $\alpha = c_1(L)$ for some ample ($\mathbb{Q}$)-line bundle $L$ over $X$. 
In order to extend the study of stability questions to a transcendental setting, recall that there is an intersection theoretic description of the Donaldson-Futaki invariant, cf. \cite{Wang} and \cite{Odaka}. As first pointed out by Berman \cite{Bermantransc}, a straightforward generalised notion of K-stability in terms of cohomology can thus be defined and a version of the YTD conjecture can be made sense of in this setting. 
The setup is explained in detail in Section \ref{test config}. 
Our main goal is to establish the following result:


\begin{thmx} \label{main conjecture semistability} 
Let $(X,\omega)$ be a compact K\"ahler manifold and let $\alpha := [\omega] \in H^{1,1}(X,\mathbb{R})$ be the corresponding K\"ahler class.
\begin{enumerate}
\item 
If the Mabuchi \emph{(K-energy)} functional is bounded from below in $\alpha$, then $(X,\alpha)$ is \emph{K-semistable} \emph{(in the generalised sense of Definition \ref{Ksemistable definition})}.
\item If the Mabuchi functional is coercive,
then $(X,\alpha)$ is uniformly K-stable \emph{(in the generalised sense of Definition \ref{definition uniform K-stability})}.  
\end{enumerate}
\end{thmx}

\noindent For precise definitions we refer to the core of the paper. As an immediate consequence of \cite[Theorem 1.1]{BB} and the above Theorem A (i) we obtain the following corollary, which is a main motivation for our work (see also Remark \ref{remark YTD} below). 

\begin{cor} \label{cor main} If the K\"ahler class $\alpha \in H^{1,1}(X,\mathbb{R})$ admits a constant scalar curvature representative, then $(X,\alpha)$ is K-semistable. 
\end{cor}

\noindent The corresponding statement in the case of a polarised manifold was first obtained by Donaldson in \cite{Donaldsoncalabi},  as an immediate consequence of the lower bound for the Calabi functional. 
See also \cite{Stoppa} and \cite{RossThomas} for related work on \emph{slope semistability}. 
The approach taken in this paper should however be compared to e.g. \cite{PRS} and  \cite{Bermantransc, Berman, BHJ1, BHJ2}, where K-semistability is derived using so called "Kempf-Ness type" formulas. By analogy to the above papers, our proof relies on establishing such formulas valid also for transcendental classes  (see Theorems B and C below), in particular relating the asymptotic slope of the K-energy along weak geodesic rays to a natural generalisation of the Donaldson-Futaki invariant. This provides a link between K-semistability (resp. uniform K-stability) and boundedness (resp. coercivity) of the Mabuchi functional, key to establishing the stability results of Theorem A. 

An underlying theme of the paper is the comparison to the extensively studied case of a polarised manifold, which becomes a "special case" in our setting. Notably, it is then known (see e.g. \cite{Bermantransc, Berman, BHJ1, BHJ2}) how to establish the sought Kempf-Ness type formulas using \emph{Deligne pairings}; a method employed by Phong-Ross-Sturm in \cite{PRS} (for further background on the Deligne pairing construction, cf. \cite{Elkik}). Unfortunately, such an approach breaks down in the case of a general K\"ahler class. 
In this paper, we circumvent this problem by a pluripotential approach, making use of a certain multivariate variant $\langle \varphi_0, \dots, \varphi_n \rangle_{(\theta_0, \dots \theta_n)}$  of the Monge-Amp\`ere energy functional, which turns out to play a role analogous to that of the Deligne pairing in arguments of the type \cite{PRS}. The Deligne pairing approach should also be compared to \cite{DonaldsonBC, TianBC} using Bott-Chern forms (see e.g. \ref{Example Deligne pairing} and \cite[Example 5.6]{Rubinsteinbook}).  

\begin{rem} \label{remark YTD} \emph{(Yau-Tian-Donaldson conjecture)} Combining Theorem A (ii) with \cite[Theorem 2.10]{DR} and a very recent result by R. Berman, T. Darvas and C. Lu \cite[Theorem 1.2]{BDL} we in fact further see that $(X,\alpha)$ is uniformly K-stable if $\alpha$ admits a constant scalar curvature representative and $\mathrm{Aut}_0(X) = \{0\}$. In case the automorphism group of $X$ is finite, the above thus confirms one direction of the Yau-Tian-Donaldson conjecture, here referring to its natural generalisation to the case of a general K\"ahler manifold, see Section \ref{concluding the main theorem}. 
\end{rem}

\smallskip

\subsection{Generalised K-semistability}
We briefly explain the framework we have in mind. As a starting point, there are natural generalisations of certain key concepts to the transcendental setting, a central notion being that of \emph{test configurations}. First recall that a test configuration for a polarised manifold $(X,L)$, in the sense of Donaldson, cf. \cite{Donaldsontoric}, is given in terms of a $\mathbb{C}^*$-equivariant degeneration $(\mathcal{X}, \mathcal{L})$ of $(X,L)$. It can be seen as an algebro-geometric way of compactifying the product $X \times \mathbb{C}^* \hookrightarrow \mathcal{X}$. Note that test configurations in the sense of Donaldson are now known, at least in the case of Fano manifolds, see \cite{LX}, to be equivalent to test configurations in the sense of Tian \cite{Tian}.

As remarked in \cite{Bermantransc}, a straightforward generalisation to the transcendental setting can be given by replacing the line bundles with $(1,1)-$cohomology classes. In the polarised setting we would thus consider $(\mathcal{X}, c_1(\mathcal{L}))$ as a "test configuration" for $(X, c_1(L))$, by simply replacing $\mathcal{L}$ and $L$ with their respective first Chern classes. The \emph{details} of how to formulate a good definition of such a generalised test configuration have, however, not yet been completely clarified. The definition given in this paper is motivated by a careful comparison to the usual polarised case, where we ensure that a number of basic but convenient tools still hold, cf. Section \ref{test config}. In particular, our notion of K-semistability coincides precisely with the usual one whenever we restrict to the case of an integral class, cf. Proposition \ref{prop comparison integral class}. 
We will refer to such generalised test configurations as \emph{cohomological}.

\begin{mydef}  \emph{(Cohomological test configuration)} A cohomological test configuration for $(X, \alpha)$ is a pair $(\mathcal{X},\mathcal{A})$ where $\mathcal{X}$ is a test configuration for $X$ \emph{(see Definition \ref{test config definition})} and $\mathcal{A} \in  H^{1,1}_{\mathrm{BC}}(\mathcal{X},\mathbb{R})^{\mathbb{C}^*}$ is a $\mathbb{C}^*$-invariant $(1,1)-$Bott-Chern cohomology class whose image under the canonical $\mathbb{C}^*$-equivariant isomorphism 
$$\mathcal{X} \setminus \mathcal{X}_0 \simeq X \times (\mathbb{P}^1 \setminus \{0\})$$ is $p_1^*\alpha$, see \eqref{equiviso}. Here $p_1: X \times \mathbb{P}^1 \rightarrow X$ denotes the first projection.  
\end{mydef}

\begin{rem}  Note that the definition is given directly over $\mathbb{P}^1$ so that we consider the Bott-Chern cohomology on a compact K\"ahler normal complex space. In the polarised case, defining a test configuration over $\mathbb{C}$ or over $\mathbb{P}^1$ is indeed equivalent, due to the existence of a natural $\mathbb{C}^*$-equivariant compactification over $\mathbb{P}^1$.
\end{rem}

\noindent In practice, it will be enough to consider the situation when the total space $\mathcal{X}$ is smooth and dominates $X \times \mathbb{P}^1$, with $\mu: \mathcal{X} \rightarrow X \times \mathbb{P}^1$ the corresponding canonical $\mathbb{C}^*$-equivariant bimeromorphic morphism. Moreover, if $(\mathcal{X}, \mathcal{A})$ is a cohomological test configuration for $(X,\alpha)$ with $\mathcal{X}$ as above, then $\mathcal{A}$ is always of the form $\mathcal{A} = \mu^*p_1^*\alpha + [D]$, for a unique $\mathbb{R}$-divisor $D$ supported on the central fiber $\mathcal{X}_0$, cf. Proposition \ref{divisor representation}. A cohomological test configuration can thus be characterised by an $\mathbb{R}$-divisor, clarifying the relationship between the point of view of $\mathbb{R}$-divisors and our cohomological approach to "transcendental K-semistability".


A straightforward generalisation of the Donaldson-Futaki invariant can be defined based on the intersection theoretic characterisation of \cite{Wang} and \cite{Odaka}. Indeed, we define the \emph{Donaldson-Futaki invariant} associated to a cohomological test configuration $(\mathcal{X},\mathcal{A})$ for $(X,\alpha)$ as the following intersection number
\begin{equation} \label{equation DF definition}
\mathrm{DF}(\mathcal{X}, \mathcal{A}) := \frac{\bar{\mathcal{S}}}{n+1} V^{-1} (\mathcal{A}^{n+1})_{\mathcal{X}} + V^{-1}(K_{\mathcal{X}/\mathbb{P}^1} \cdot \mathcal{A}^n)_{\mathcal{X}},
\end{equation}
computed on the (compact) total space $\mathcal{X}$. Here $V$ and $\bar{\mathcal{S}}$ are cohomological constants denoting the K\"ahler volume and mean scalar curvature of $(X,\alpha)$ respectively. 

Finally, we say that $(X,\alpha)$ is \emph{K-semistable} if $\mathrm{DF}(\mathcal{X}, \mathcal{A}) \geq 0$ for all cohomological test configurations $(\mathcal{X}, \mathcal{A})$ for $(X,\alpha)$ where the class $\mathcal{A}$ is \emph{relatively K\"ahler}, i.e. there is a K\"ahler form $\beta$ on $\mathbb{P}^1$ such that $\mathcal{A} + \pi^*\beta$ is K\"ahler on $\mathcal{X}$. Generalised notions of (uniform) K-stability are defined analogously.

\subsection{Transcendental Kempf-Ness type formulas}

As previously stated, a central part of this paper consists in establishing a Kempf-Ness type formula connecting the \emph{Donaldson-Futaki invariant} (in the sense of \eqref{equation DF definition} above) with the \emph{asymptotic slope of the K-energy} along certain weak geodesic rays. In fact, we first prove the following result, which is concerned with asymptotics of a certain multivariate analogue of the Monge-Amp\`ere energy, cf. Section \ref{multivariate energy functional} for its definition. It turns out to be very useful for establishing a similar formula for the K-energy (cf. Remark \ref{introremark}), but may also be of independent interest. 

\noindent In what follows, we will work on the level of potentials and refer the reader to Section \ref{proof B} for precise definitions. 

\begin{thmx} 
Let $X$ be a compact K\"ahler manifold of dimension $n$ and let $\theta_i$, $0 \leq i \leq n$, be closed $(1,1)$-forms on $X$. Let $(\mathcal{X}_i, \mathcal{A}_i)$ be cohomological test configurations for $(X,\alpha_i)$, where $\alpha_i := [\theta_i] \in H^{1,1}(X,\mathbb{R})$. Then, for each collection of smooth rays $(\varphi_i^t)_{t \geq 0}$, $\mathcal{C}^{\infty}$-compatible with $(\mathcal{X}_i, \mathcal{A}_i)$ respectively, the asymptotic slope of the multivariate energy functional $\langle \cdot, \dots, \cdot \rangle := \langle \cdot, \dots, \cdot \rangle_{(\theta_0, \dots, \theta_n)}$ is well-defined and satisfies
\begin{equation} \label{multivariate slope}
\frac{\langle \varphi_0^t, \dots, \varphi_n^t \rangle}{t} \longrightarrow (\mathcal{A}_0 \cdot \dots \cdot \mathcal{A}_n)
\end{equation}
as $t \rightarrow +\infty$. See \eqref{definition intersection number} for the definition of the above intersection number. 
\end{thmx}

\begin{rem} \label{introremark} In the setting of Hermitian line bundles, the above multivariate energy functional naturally appears as the difference (or quotient) of metrics on Deligne pairings.  
Moreover, note that the above theorem applies to e.g. Aubin's $\mathrm{J}$-functional, the Monge-Amp\`ere energy functional $\mathrm{E}$ and its 'twisted' version $\mathrm{E}^{\ric \omega}$ but \emph{not} to the K-energy $\mathrm{M}$. 
Indeed, the expression for $\mathrm{M}(\varphi_t)$ on the above form $\langle \varphi_0^t, \dots, \varphi_n^t \rangle_{(\theta_0, \dots, \theta_n)}$ involves the metric $\log(\omega + dd^c\varphi_t)^n$ on the relative canonical bundle $K_{\mathcal{X}/\mathbb{P}^1}$, which blows up close to $\mathcal{X}_0$, cf. Section \ref{proof A}.
As observed in \cite{BHJ2}, it is however possible to find functionals of the above form that 'approximate' $\mathrm{M}$ in the sense that their asymptotic slopes coincide, up to an explicit correction term that vanishes precisely when the central fiber $\mathcal{X}_0$ is reduced. 
This is a key observation.
\end{rem} 

\noindent We further remark that such a formula \eqref{multivariate slope} cannot be expected to hold unless the test configurations $(\mathcal{X}_i, \mathcal{A}_i)$ and the rays $(\varphi_i^t)$ are compatible in a certain sense. This is the role of the notion of $\mathcal{C}^{\infty}$-compatibility (as well as the $\mathcal{C}^{1,1}_{\mathbb{C}}$-compatibility used in Theorem C below). These notions may seem technical, but in fact mimic the case of a polarised manifolds, where the situation is well understood in terms of extension of metrics on line bundles, cf. Section \ref{smoothcompatibility}.


As a further important consequence of the above Theorem B we deduce that if $(\mathcal{X}, \mathcal{A})$ is a relatively K\"ahler cohomological test configuration for $(X,\alpha)$, then for each smooth ray $(\varphi_t)_{t \geq 0}$, $\mathcal{C}^{\infty}$-compatible with $(\mathcal{X},\mathcal{A})$, we have the inequality 
\begin{equation} \label{inequality}
\lim_{t \rightarrow +\infty} \frac{\mathrm{M}(\varphi_{t})}{t} \leq \mathrm{DF}(\mathcal{X}, \mathcal{A}).
\end{equation}  
This is the content of Theorem  \ref{weak Theorem C}, and should be compared to the discussion in the introduction of \cite{PRS}. As an important special case, this inequality can be seen to hold in the case of a weak geodesic ray associated to the given test configuration $(\mathcal{X}, \mathcal{A})$, cf. Section \ref{smoothcompatibility} for its construction. The inequality \eqref{inequality} is moreover enough to conclude the proof of Theorem A, as explained in Section \ref{concluding the main theorem}.

Using ideas from \cite{BHJ2} adapted to the present setting, we may further improve on formula \eqref{inequality} and compute the precise asymptotic slope of the K-energy. 
In this context, it is natural to consider the \emph{non-Archimedean Mabuchi functional} 
$$
\mathrm{M}^{\mathrm{NA}}(\mathcal{X}, \mathcal{A}) :=  \mathrm{DF}(\mathcal{X}, \mathcal{A}) + V^{-1}((\mathcal{X}_{0,\mathrm{red}} - \mathcal{X}_0) \cdot \mathcal{A}^n)_{\mathcal{X}},
$$

\noindent cf. \cite{BHJ1} and \cite{BHJ2} for an explanation of the terminology. 
It is a modification of the Donaldson-Futaki invariant which is \emph{homogeneous under finite base change}, and which satisfies $\mathrm{M}^{\mathrm{NA}}(\mathcal{X}, \mathcal{A}) \leq \mathrm{DF}(\mathcal{X}, \mathcal{A})$ with equality precisely when the central fiber $\mathcal{X}_0$ is reduced. We then have the following result, special cases of which have been obtained by previous authors in various different situations and generality.

\begin{thmx} \label{Reduced Theorem B} 
Let $(\mathcal{X}, \mathcal{A})$ be a smooth, relatively K\"ahler cohomological test configuration for $(X,\alpha)$ dominating $X \times \mathbb{P}^1$. For each subgeodesic ray $(\varphi_t)_{t \geq 0}$, $\mathcal{C}^{1,1}_{\mathbb{C}}$-compatible with $(\mathcal{X},\mathcal{A})$, the following limit is well-defined and satisfies
\begin{equation*}
\frac{\mathrm{M}(\varphi_{t})}{t} \longrightarrow \mathrm{M}^{\mathrm{NA}}(\mathcal{X}, \mathcal{A}),
\end{equation*}
as $t \rightarrow +\infty$. In particular, this result holds for the weak geodesic ray associated to $(\mathcal{X},\mathcal{A})$, constructed in Lemma \ref{weak geodesic construction}.
\end{thmx}



%
\noindent For polarised manifolds $(X,L)$ and \emph{smooth} subgeodesic rays $(\varphi_t)_{t \geq 0}$, this precise result was proven in \cite{BHJ2} using Deligne pairings, as pioneered by Phong-Ross-Sturm in \cite{PRS} (cf. also Paul-Tian \cite{PaulTian} \cite{PaulTian2}). A formula in the same spirit has also been obtained for the so called \emph{Ding functional} when $X$ is a Fano variety, see \cite{Berman}. However, it appears as though no version of this result was previously known in the case of non-polarised manifolds. 




\subsection{Structure of the paper}

In Section \ref{Preliminaries} we fix our notation for energy functionals and subgeodesic rays. In particular, we introduce the multivariate energy functionals $\langle \cdot, \dots, \cdot \rangle_{(\theta_0, \dots, \theta_n)}$, which play a central role in this paper. 
In Section \ref{test config} we introduce our generalised notion of \emph{cohomological test configurations} and K-semistability. In the case of a polarised manifold $(X,L)$, we compare this notion to the usual algebraic one.
We also discuss classes of cohomological test configurations for which it suffices to test K-semistability, and establish a number of basic properties. 
In Section \ref{proof B} we discuss transcendental Kempf-Ness type formulas and prove Theorem B. This involves introducing natural \emph{compatibility conditions} between a ray $(\varphi_t)$ and a cohomological test configuration $(\mathcal{X}, \mathcal{A})$ for $(X,\alpha)$. As a useful special case, we discuss the weak geodesic ray associated to $(\mathcal{X}, \mathcal{A})$. 
In Section \ref{proof A} we finally apply Theorem B to yield a weak version of Theorem C, from which we in turn deduce our main result, Theorem A. By an immediate adaptation of techniques from \cite{BHJ2} we then compute the precise asymptotic slope of the Mabuchi functional, thus establishing the full Theorem C. 


\subsection{Acknowledgements}
Since the first version of this paper was made available, R. Dervan and J. Ross, independently, used similar methods to establish Theorem A (see \cite{DervanRoss}). They are further able to establish K-stability of $(X,\omega)$ whenever the automorphism group $\mathrm{Aut}(X,\omega)$ is discrete. We are grateful to them for helpful discussions on the topic of this paper and related questions. 

Moreover, it is a pleasure to thank my thesis advisors S\'ebastien Boucksom and Vincent Guedj, as well as Robert Berman and Ahmed Zeriahi for many helpful discussions and suggestions.

\medskip

\section{Preliminaries}  \label{Preliminaries}

\subsection{Notation and basic definitions} 
Let $X$ be a compact complex manifold of $\mathrm{dim}_{\mathbb{C}}X = n$ equipped with a given K\"ahler form $\omega$, i.e. a smooth real closed positive $(1,1)$-form on $X$. Denote the K\"ahler class $[\omega] \in H^{1,1}(X,\mathbb{R})$ by $\alpha$. 

In order to fix notation, let $\ric{\omega} = -dd^c\log \omega^n$ be the Ricci curvature form, where $dd^c := \frac{\sqrt{-1}}{2\pi}\partial\bar{\partial}$ is normalised so that $\ric{\omega}$ represents the first Chern class $c_1(X)$. Its trace
\begin{equation*}
\mathcal{S}(\omega) := n \frac{\ric{\omega} \wedge \omega^{n-1}}{\omega^n}
\end{equation*} 
is the scalar curvature of $\omega$. The mean scalar curvature is the cohomological constant given by
\begin{equation}
\bar{\mathcal{S}} := V^{-1} \int_X \mathcal{S}(\omega) \; \omega^n =  n\frac{\int_X c_1(X) \cdot \alpha^{n-1}}{ \int_X \alpha^n} := n  \frac{(c_1(X) \cdot  \alpha^{n-1})_X}{( \alpha^n)_X},
\end{equation}
where $V := \int_X \omega^n := ( \alpha^n)_X  $ is the K\"ahler volume. 
We say that $\omega$ is a constant scalar curvature K\"ahler (cscK) metric\footnote{By a standard abuse of notation we identify a Hermitian metric $h$ with its associated $(1,1)$-form $\omega = \omega_h$ and refer to $\omega$ as the "metric".} if $S(\omega)$ is constant (equal to $\bar{\mathcal{S}}$) on X. 

Throughout the paper we work on the level of potentials, using the notation of quasi-plurisubharmonic (quasi-psh) functions. To this end, we let $\theta$ be a closed $(1,1)$-form on $X$ and denote, as usual, by $\mathrm{PSH}(X,\theta)$ the space of $\theta$-psh functions $\varphi$ on $X$, i.e. the set of functions that can be locally written as the sum of a smooth and a plurisubharmonic function, and such that $$\theta_{\varphi} := \theta + dd^c\varphi \geq 0 $$
in the weak sense of currents. In particular, if $\omega$ is our fixed K\"ahler form on $X$, then we write 
$$
\mathcal{H}_{\omega}  := \{ \varphi \in \mathcal{C}^{\infty}(X) : \omega_{\varphi} := \omega + dd^c\varphi > 0 \} \subset \mathrm{PSH}(X,\omega)
$$ 
for the space of K\"ahler potentials on $X$. As a subset of $\mathcal{C}^{\infty}(X)$ it is convex and consists of strictly $\omega$-psh functions. It has been extensively studied (for background we refer the reader to e.g. \cite{LNM} and references therein).

Recall that a $\theta$-psh function is always upper semi-continuous (usc) on $X$, thus bounded from above by compactness. Moreover, if $\varphi_i \in \mathrm{PSH}(X,\theta) \cap L^{\infty}_{\mathrm{loc}}$, $1 \leq i \leq p \leq n$,  it follows from the work of Bedford-Taylor \cite{BTcontinuouspotentials,BT} that we can give meaning to the product $\bigwedge_{i = 1}^p (\theta + dd^c\varphi_i)$, which then defines a closed positive $(p,p)$-current on $X$. As usual, we then define the \emph{Monge-Amp\`ere measure} as the following probability measure, given by the top wedge product 
\begin{equation*}
\mathrm{MA}(\varphi) := V^{-1}(\omega + dd^c\varphi)^{n}.
\end{equation*}


\subsection{Energy functionals and a Deligne functional formalism} \label{multivariate energy functional} 
We now introduce the notation for energy functionals that we will use. Let $\varphi_i \in \mathrm{PSH}(X,\omega) \cap L^{\infty}_{\mathrm{loc}}$. The \emph{Monge-Amp\`ere energy functional} (or \emph{Aubin-Mabuchi functional}) $\mathrm{E} := \mathrm{E}_{\omega}$ is defined by 
\begin{equation*}
\mathrm{E}(\varphi) := \frac{1}{n+1}\sum_{j=0}^n V^{-1} \int_X \varphi (\omega + dd^c\varphi)^{n-j} \wedge \omega^j
\end{equation*}
(such that $\mathrm{E}' = \mathrm{MA}$ and $\mathrm{E}(0) = 0$). Similarily, if $\theta$ is any closed  $(1,1)$-form, we define a functional $\mathrm{E}^{\theta} := \mathrm{E}_{\omega}^{\theta}$ by
\begin{equation*}
\mathrm{E}^{\theta}(\varphi) :=  \sum_{j=0}^{n-1} V^{-1} \int_X \varphi (\omega + dd^c\varphi)^{n-j-1} \wedge \omega^j \wedge \theta,
\end{equation*}
and we will also have use for the Aubin $\mathrm{J}$-functional $\mathrm{J}: \mathrm{PSH}(X,\omega) \cap L^{\infty}_{\mathrm{loc}} \rightarrow \mathbb{R}_{\geq 0}$ defined by $$\mathrm{J}(\varphi) := V^{-1} \int_X \varphi \; \omega^n - \mathrm{E}(\varphi),$$ whose asymptotic slope along geodesic rays is comparable with the \emph{minimum norm} of a test configuration (see \cite{Dervanuniform, BHJ1}). 

More generally, it is possible to define a natural \emph{multivariate} version of the Monge-Amp\`ere energy, of which all of the above functionals are special cases. The construction builds on that of the \emph{Deligne pairing}, which is a powerful and general technique from algebraic geometry that we here apply to our specific setting in K\"ahler geometry. We refer the interested reader to \cite{Elkik, Zhangdelignepairings, Moriwaki} for a general treatment of Deligne pairings, as well as to \cite{PRS, Berman, BHJ2} for more recent applications related to K-stability. Now let $\theta_0, \dots, \theta_n$ be closed $(1,1)$-forms on $X$. Motivated by corresponding properties for the \emph{Deligne pairing} (cf. e.g. \cite{Berman}, \cite{Elkik} for background) we would like to consider a \emph{functional} $\langle \cdot, \dots, \cdot \rangle_{(\theta_0,\dots,\theta_n)}$ on the space $\mathrm{PSH}(X,\theta_0) \cap L^{\infty}_{\mathrm{loc}} \times \dots \times \mathrm{PSH}(X,\theta_n) \cap L^{\infty}_{\mathrm{loc}}$ ($n + 1$ times) that is
 \begin{itemize}
 \item symmetric, i.e. for any permutation $\sigma$ of the set $\{0,1, \dots,n\}$, we have
\begin{equation*}
\langle \varphi_{\sigma(0)}, \dots, \varphi_{\sigma(n)} \rangle_{(\theta_{\sigma(0)}, \dots, \theta_{\sigma(n)})} = \langle \varphi_0, \dots, \varphi_n \rangle_{(\theta_0,\dots,\theta_n)}. 
\end{equation*}  
 \item  if $\varphi_0'$ is another $\theta_i$-psh function in $\mathrm{PSH}(X,\theta) \cap L^{\infty}_{\mathrm{loc}}$, then we have a 'change of function' property
\begin{equation*}
\langle \varphi_0', \varphi_1 \dots, \varphi_n \rangle - \langle \varphi_0, \varphi_1 \dots, \varphi_n \rangle  = 
\end{equation*}
\begin{equation*}
 \int_X (\varphi_0' - \varphi_0) \; (\omega_1 + dd^c\varphi_1) \wedge \dots \wedge (\omega_n + dd^c\varphi_n). 
\end{equation*}
\end{itemize}  
Demanding that the above properties hold necessarily leads to the following definition of \emph{Deligne functionals}, that will provide a useful terminology for this paper; 

\begin{mydef} \label{Multivariate energy def} Let $\theta_0, \dots, \theta_n$ be closed $(1,1)$-forms on $X$. Define a \emph{multivariate energy functional} $\langle \cdot, \dots, \cdot \rangle_{(\theta_0,\dots,\theta_n)}$ on the space $\mathrm{PSH}(X,\theta_0) \cap L^{\infty}_{\mathrm{loc}} \times \dots \times \mathrm{PSH}(X,\theta_n) \cap L^{\infty}_{\mathrm{loc}}$ \emph{($n + 1$ times)} by 
\begin{equation*}
 \langle\varphi_0, \dots, \varphi_n \rangle_{(\theta_0,\dots,\theta_n)} :=  \int_X \varphi_0 \; (\theta_1 + dd^c\varphi_1) \wedge \dots \wedge (\theta_n + dd^c\varphi_n)
\end{equation*} 
\begin{equation*}
 + \int_X \varphi_1 \; \theta_0 \wedge (\theta_2 + dd^c\varphi_2) \wedge \dots \wedge (\theta + dd^c\varphi_n) +  \dots + \int_X \varphi_n \; \theta_0 \wedge \dots \wedge \theta_{n-1}.
\end{equation*}
\end{mydef} 

\begin{rem} The multivariate energy functional $\langle \cdot, \dots, \cdot \rangle_{(\theta_0,\dots,\theta_n)}$ can also be defined on $\mathcal{C}^{\infty}(X) \times \dots \times \mathcal{C}^{\infty}(X)$ by the same formula. In Sections \ref{proof B} and \ref{proof A} it will be interesting to consider both the smooth case and the case of locally bounded $\theta_i$-psh functions.
\end{rem}

\noindent Using integration by parts one can check that this functional is indeed symmetric.

\begin{prop}
The functional $\langle \cdot, \dots, \cdot \rangle_{(\theta_0,\dots,\theta_n)}$ is symmetric. 
\end{prop}

\begin{proof}
Since every permutation is a composition of transpositions it suffices to check the sought symmetry property for transpositions $\sigma := \sigma_{j,k}$ exchanging the position of $j,k \in \{0,1,\dots,n\}$. Suppose for simplicity of notation that $j < k$ and write $\theta_i^t := \theta_i + dd^c\varphi_i$. A straightforward computation then yields
\begin{equation*}
\langle \varphi_{0}, \dots,\varphi_j, \varphi_k, \dots \varphi_{n} \rangle_{(\theta_{0}, \dots, \theta_j, \theta_k, \dots \theta_{n})} - \langle \varphi_{0}, \dots,\varphi_k, \varphi_j, \dots \varphi_{n} \rangle_{(\theta_{0}, \dots, \theta_k, \theta_j, \dots \theta_{n})} =
\end{equation*}  
$$
 = \int_X \varphi_jdd^c\varphi_k \wedge \Theta_{j,k} - \int_X \varphi_kdd^c\varphi_j \wedge \Theta_{j,k} = 0,  
$$
where in the last step we used integration by parts and 
$$
\Theta_{j,k} := \theta_0 \wedge \dots \wedge \theta_{j-1} \wedge \theta_{j+1}^t \wedge \dots \theta_{k-1}^t \wedge \theta_{k+1}^t \wedge \theta_n^t,
$$ 
(with factors $\theta_j$ and $\theta_k^t$ omitted). The case $j > k$ follows in the exact same way, with obvious modifications to the above proof.
\end{proof}

\begin{ex} \label{Example Deligne pairing} As previously remarked, note that the above functionals can be written using the Deligne functional formalism. Indeed, if $\theta$ is a closed $(1,1)$-form on $X$, $\omega$ is a K\"ahler form on $X$ and $\varphi$ is an $\omega$-psh function on $X$, then
\begin{equation*}
\mathrm{E}(\varphi) = \frac{1}{n+1} V^{-1} \langle \varphi, \dots, \varphi \rangle_{(\omega, \dots, \omega)} \;, \; \; \mathrm{E}^{\theta}(\varphi) =  V^{-1} \langle 0,\varphi, \dots, \varphi \rangle_{(\theta, \omega, \dots, \omega)}
\end{equation*} 
and 
\begin{equation*}
\mathrm{J}(\varphi)  =  V^{-1} \langle \varphi,0, \dots, 0 \rangle_{(\omega, \dots, \omega)} - \mathrm{E}(\varphi).
\end{equation*}
Compare also \cite[Example 5.6]{Rubinsteinbook} on Bott-Chern forms. 
\end{ex}

 
\subsection{Subgeodesic rays}
Let  $(\varphi_t)_{t \geq 0} \subset \mathrm{PSH}(X,\omega)$ be a ray of $\omega$-psh functions.  Following a useful point of view of Donaldson \cite{Donaldsontoric} and Semmes \cite{Semmes}, there is a basic correspondence between the family $(\varphi_t)_{t \geq 0}$ and an associated $S^1$-invariant function $\Phi$ on  $X \times \bar{\Delta}^*$, where $\bar{\Delta}^*\subset \mathbb{C}$ denotes the pointed unit disc. We denote by $\tau$ the coordinate on $\Delta$. Explicitly, the correspondence is given by
\begin{equation*}
\Phi(x,e^{-t+is}) = \varphi^t(x),
\end{equation*}
where the sign is chosen so that $t \rightarrow +\infty$ corresponds to $\tau := e^{-t+is} \rightarrow 0$. The function $\Phi$ restricted to a fiber $X \times \{\tau\}$ thus corresponds precisely to $\varphi_t$ on $X$. In the direction of the fibers we thus have $p_1^*\omega + dd^c_{x}\Phi \geq 0$ (in the sense of currents, letting  $p_1: X \times \Delta \rightarrow X$ denote the first projection).

We will use the following standard terminology, motivated by the extensive study of (weak) geodesics in the space $\mathcal{H}$, 
see e.g. \cite{Blocki13}, \cite{Chen00}, \cite{Darvas14}, \cite{Donaldsontoric}, \cite{Semmes}. 
\begin{mydef} \label{subgeodesic definition}
We say that $(\varphi_t)_{t \geq 0}$ is a \emph{subgeodesic ray} if the associated $S^1$-invariant function $\Phi$ on $X \times \bar{\Delta}^*$ is $p_1^*\omega$-psh. Furthermore, a locally bounded family of functions $(\varphi_t)_{t \geq 0}$ in $\mathrm{PSH}(X,\omega)$ is said to be a \emph{weak geodesic ray} if the associated $S^1$-invariant function $\Phi \in \mathrm{PSH}(X \times \bar{\Delta}^*, p_1^*\omega)$ 
satisfies
\begin{equation*}
(p_1^*\omega + dd^c_{(x,\tau)}\Phi)^{n+1} = 0
\end{equation*}
on $X \times \Delta^*$. 
\end{mydef}

\begin{mydef}
Viewing the family $(\varphi_t)_{t \geq 0}$ as a map $(0,+\infty) \rightarrow \mathrm{PSH}(X,\omega)$, we say that $(\varphi^t)_{t \geq 0}$ is continuous (resp. locally bounded, smooth) if the corresponding $S^1$-invariant function $\Phi$ is continuous \emph{(}resp. locally bounded, smooth\emph{)}. 
\end{mydef}

\noindent The existence of geodesics with bounded Laplacian was proven by Chen \cite{Chen00} with complements by Blocki \cite{Blocki13}, see also e.g. \cite{Darvas14},  \cite{DL12}. We will refer to such a geodesic as being \emph{$\mathcal{C}^{1,1}_{\mathbb{C}}$-regular},
cf. Lemma \ref{weak geodesic construction} below. 

\begin{mydef}\label{Coneone def} 
We say that a function $\varphi$ is \emph{$\mathcal{C}^{1,1}_{\mathbb{C}}$-regular} if $dd^c\varphi \in L^{\infty}_{\mathrm{loc}}$, and we set $\mathcal{H}^{1,1}_{\mathbb{C}}:= \mathrm{PSH}(X,\omega) \cap \mathcal{C}^{1,1}_{\mathbb{C}}$.
\end{mydef}

\noindent Recall that a $\mathcal{C}^{1,1}_{\mathbb{C}}$-regular function is automatically $\mathcal{C}^{1,a}$-regular for all $0 < a < 1$. On the other hand, this condition is weaker than $\mathcal{C}^{1,1}$-regularity (i.e. bounded real Hessian). 

\medskip

\subsection{Second order variation of Deligne functionals}
We have the following identity for the second order variations of the multivariate energy functional $\langle \cdot, \dots, \cdot \rangle_{(\theta_0, \dots, \theta_n)}$.  

\begin{prop} \label{second order variation multilinear energy} Let $\theta_0, \dots, \theta_n$ be closed $(1,1)$-forms on $X$ and let $(\varphi_i^t)_{t \geq 0}$ be a smooth ray of smooth functions. Let $\tau := e^{-t + is}$ and consider the reparametrised ray $(\varphi_i^{\tau})_{\tau \in \bar{\Delta}^*}$. Denoting by $\Phi_i$ the corresponding $S^1$-invariant function on $X \times \Delta^*$, we have 
\begin{equation*}
dd^c_{\tau} \langle \varphi_0^{\tau}, \dots, \varphi_n^{\tau} \rangle_{(\theta_0, \dots, \theta_n)} = \int_X (p_1^*\theta_0 + dd^c_{(x,\tau)}\Phi_0) \wedge \dots \wedge (p_1^*\theta_n + dd^c_{(x,\tau)}\Phi_n) 
\end{equation*}
$$
$$
where $\int_X$ denotes fiber integration, i.e. pushforward of currents.
\end{prop}
\begin{proof}
The result follows from a computation relying on integration by parts and is an immediate adaptation of for instance \cite[Proposition 6.2]{BBGZ}. 
\end{proof}

\noindent As a particular case of the above, we obtain the familiar formulas for the second order variation of $E$ and $E^{\theta}$, given by
\begin{equation*} 
dd^c_{\tau}\mathrm{E}(\varphi_{\tau}) = \frac{1}{n + 1} V^{-1} \int_X (p_1^*\omega + dd_{(x,\tau)}^c\Phi)^{n+1} 
\end{equation*}
and
\begin{equation*} 
dd^c_{\tau}\mathrm{E}^{\theta}(\varphi_{\tau}) = V^{-1} \int_X (p_1^*\omega + dd_{(x,\tau)}^c\Phi)^{n} \wedge \theta 
\end{equation*}
respectively. 
In particular, note that $\mathrm{E}(\varphi_{\tau}) := \mathrm{E} \circ \Phi$ is a subharmonic function on $\bar{\Delta}^*$. The function $t \mapsto \mathrm{E}(\varphi^{\tau})$ is \emph{affine} along weak geodesics, and \emph{convex} along subgeodesics.


\subsection{The K-energy and the Chen-Tian formula}
Let $\omega$ be a K\"ahler form on $X$ and consider any path $(\varphi_t)_{t \geq 0}$ in the space $\mathcal{H}$ of K\"ahler potentials on $X$. The \emph{Mabuchi functional} (or \emph{K-energy}) $\mathrm{M}: \mathcal{H} \rightarrow \mathbb{R}$ is then defined by its Euler-Lagrange equation
\begin{equation*} 
\frac{d}{dt} \mathrm{M}(\varphi_t) =  V^{-1} \int_X \dot{\varphi}_t(\mathcal{S}(\omega_{\varphi_t}) - \bar{\mathcal{S}})\; \omega_{\varphi_t}^n
\end{equation*} 
It is indeed independent of the path chosen and the critical points of the Mabuchi functional are precisely the cscK metrics, when they exist. By \emph{the Chen-Tian formula} \cite{Chen} it is possible to write the Mabuchi functional as a sum of an "energy" and an "entropy" part. More precisely, with our normalisations we have
\begin{equation} \label{Chens formula} 
\mathrm{M}(\varphi) = \left(\bar{\mathcal{S}} \mathrm{E}(\varphi) - \mathrm{E}^{\ric{\omega}}(\varphi)\right) + V^{-1} \int_X \log \left( \frac{(\omega + dd^c\varphi)^n}{\omega^n} \right) (\omega + dd^c\varphi)^n,
\end{equation}
where the latter term is the \emph{relative entropy} of the probability measure $\mu := \omega_{\varphi}^n/V$ with respect to the reference measure $\mu_0 := \omega^n/V$. Recall that the entropy takes values in $[0, +\infty]$ and is finite if $\mu/\mu_0$ is bounded. It can be seen to be always lower semi-continuous (lsc) in $\mu$.

%

 Following Chen \cite{Chen} (using the formula \eqref{Chens formula}) we will often work with the extension $\mathrm{M}:\mathcal{H}^{1,1}_{\mathbb{C}}\rightarrow \mathbb{R}$ of the Mabuchi functional to the space of $\omega$-psh functions with bounded Laplacian. This is a natural setting to consider, since weak geodesic rays with bounded Laplacian are known to always exist, cf. \cite{Chen00, Blocki13, Darvas14, DL12} as well as Lemma \ref{weak geodesic construction}.

For later use, we also state the following definition.

\begin{mydef} \label{definition coercivity}
The Mabuchi K-energy functional is said to be \emph{coercive} if there are constants $\delta, C > 0$ such that 
$$
\mathrm{M}(\varphi) \geq \delta \mathrm{J}(\varphi) - C 
$$
for all $\varphi \in \mathcal{H}$. 
\end{mydef}

We further recall that the Mabuchi functional is convex along weak geodesic rays, as was recently established by \cite{BB}, see also \cite{ChenLiPaun}. As a consequence of this convexity, the Mabuchi functional is bounded from below (in the given K\"ahler class) whenever $\alpha$ contains a cscK metric, see \cite{Donaldsonscalarcurvature}, \cite{ChiLi} for a proof in the polarised case and \cite{BB} for the general K\"ahler setting. 

\medskip

\section{Cohomological test configurations and K-semistability} \label{test config}
\noindent In this section we introduce a natural generalised notion of test configurations and K-semistability of $(X,\alpha)$ that has meaning even when the class $\alpha \in H^{1,1}(X,\mathbb{R})$ is non-integral (or non-rational), i.e. when $\alpha$ is not necessarily of the form $c_1(L)$ for some ample ($\mathbb{Q}$)-line bundle $L$ on $X$. As remarked by Berman in \cite{Bermantransc}, it is natural to generalise the notion of test configuration in terms of cohomology classes.  In the polarised setting the idea is to consider $(\mathcal{X}, c_1(\mathcal{L}))$ as a "test configuration" for $(X, c_1(L))$, by simply replacing $\mathcal{L}$ and $L$ with their respective first Chern classes. This approach is motivated in detail below. Moreover, a number of basic and useful properties will be established, and throughout, this generalisation will systematically be compared to the original notion of algebraic test configuration $(\mathcal{X}, \mathcal{L})$ for $(X,L)$, introduced by Donaldson in \cite{Donaldsontoric}.

\begin{rem} Much of the following exposition goes through even when the cohomology class $\alpha$ is not K\"ahler. Unless explicitly stated otherwise, we thus assume that $\alpha = [\theta]$ for some closed $(1,1)$-form $\theta$ on $X$.
\end{rem} 

\subsection{Test configurations for $X$} \label{compactification} 
We first introduce the notion of test configuration $\mathcal{X}$ for $X$, working directly over $\mathbb{P}^1$. For the sake of comparison, recall the usual concept of test configuration for polarised manifolds, see e.g. \cite{BHJ1} and \cite{Gabornotes}. In what follows, we refer to \cite{Fischer} for background on normal complex spaces. 
\begin{mydef} \label{test config definition} 
A \emph{test configuration $\mathcal{X}$ for $X$} consists of  
\begin{itemize}
\item a normal \emph{compact} K\"ahler complex space $\mathcal{X}$ with a flat \emph{(i.e. surjective)} morphism $\pi: \mathcal{X} \rightarrow \mathbb{P}^1$
\item a $\mathbb{C}^*$-action $\lambda$ on $\mathcal{X}$ lifting the canonical action on $\mathbb{P}^1$
\item a $\mathbb{C}^*$-equivariant isomorphism 
\begin{equation} \label{equiviso}
\mathcal{X} \setminus \mathcal{X}_0 \simeq X \times (\mathbb{P}^1 \setminus \{0\}).
\end{equation}
\end{itemize}
\end{mydef}


\noindent The isomorphism \ref{equiviso} gives an open embedding of $X \times (\mathbb{P}^1 \setminus \{0\})$ into $\mathcal{X}$, hence induces a canonical $\mathbb{C}^*$-equivariant bimeromorphic map $\mu: \mathcal{X} \dashrightarrow X \times \mathbb{P}^1$. We say that $\mathcal{X}$ dominates $X \times \mathbb{P}^1$ if the above bimeromorphic map $\mu$ is a morphism.
Taking $\mathcal{X}'$ to be the normalisation of the graph of $\mathcal{X} \dashrightarrow X \times \mathbb{P}^1$ we obtain a $\mathbb{C}^*$-equivariant bimeromorphic morphism $\rho: \mathcal{X}' \rightarrow \mathcal{X}$ with $\mathcal{X}'$ normal and dominating $X \times \mathbb{P}^1$. 
In the terminology of \cite{BHJ1} such a morphism $\rho$ 
is called a \emph{determination} of $\mathcal{X}$. In particular, a determination of $\mathcal{X}$ always exists. By the above considerations we will often, up to replacing $\mathcal{X}$ by $\mathcal{X}'$, be able to assume that the given test configuration for $X$ dominates $X \times \mathbb{P}^1$. 

Moreover, any test configuration $\mathcal{X}$ for $X$ can be dominated by a smooth test configuration $\mathcal{X}'$ for $X$ (where we may even assume that $\mathcal{X}'_0$ is a divisor of simple normal crossings). Indeed, by Hironaka (see \cite[Theorem 45]{Kollar} for the precise statement concerning normal complex spaces) there is a $\mathbb{C}^*$-equivariant proper bimeromorphic map $\mu: \mathcal{X}' \rightarrow \mathcal{X}$, with $\mathcal{X}'$ smooth, such that $\mathcal{X}_0'$ has simple normal crossings and $\mu$ is an isomorphism outside of the central fiber $\mathcal{X}_0$. 

As a further consequence of the isomorphism  \eqref{equiviso}, note that if $\Phi$ is a function on $\mathcal{X}$, then its restriction to each fibre $\mathcal{X}_{\tau} \simeq X$,  $\tau \in \mathbb{P}^1 \setminus \{0\}$ identifies with a function on $X$. The function $\Phi$ thus gives rise to a family of functions $(\varphi_t)_{t \geq 0}$ on $X$, recalling our convention of reparametrising so that $t := - \log |\tau|$.

\begin{rem} When $X$ is projective (hence algebraic), the GAGA principle shows that the usual (i.e. algebraic, and normal) test configurations of $X$ correspond precisely to the test configurations (in our sense of Definition \ref{test config definition}) with $\mathcal{X}$ projective. 
\end{rem}

\subsection{Cohomological test configurations for $(X,\alpha)$}
We now introduce a natural generalisation of the usual notion of \emph{algebraic test configuration} $(\mathcal{X}, \mathcal{L})$ for a polarised manifold $(X,L)$. This following definition involves the Bott-Chern cohomology on normal complex spaces, i.e. the space of locally $dd^c$-exact $(1,1)$-forms (or currents) modulo globally $dd^c$-exact $(1,1)$-forms (or currents). The Bott-Chern cohomology is finite dimensional and the cohomology classes can be pulled back. Moreover, $H^{1,1}_{\mathrm{BC}}(\mathcal{X}, \mathbb{R})$ coincides with the usual Dolbeault cohomology $H^{1,1}(\mathcal{X},\mathbb{R})$ whenever $\mathcal{X}$ is smooth. See e.g. \cite{LNM2} for background. 

\begin{mydef}  \label{cohomological test config definition} 
A \emph{cohomological test configuration for $(X, \alpha)$} is a pair $(\mathcal{X},\mathcal{A})$ where $\mathcal{X}$ is a test configuration for $X$ and $\mathcal{A} \in  H^{1,1}_{\mathrm{BC}}(\mathcal{X},\mathbb{R})^{\mathbb{C}^*}$ is a $\mathbb{C}^*$-invariant $(1,1)-$Bott-Chern cohomology class whose image under the $\mathbb{C}^*$-equivariant isomorphism $$\mathcal{X} \setminus \mathcal{X}_0 \simeq X \times (\mathbb{P}^1 \setminus \{0\}).
$$
is $p_1^*\alpha$. Here $p_1: X \times \mathbb{P}^1 \rightarrow X$ is the first projection.
\end{mydef}

\begin{mydef} We say that a test configuration $(\mathcal{X}, \mathcal{A})$ for $(X,\alpha)$  is \emph{smooth} if the total space $\mathcal{X}$ is smooth. In case $\alpha \in H^{1,1}(X,\mathbb{R})$ is K\"ahler, we say that  $(\mathcal{X}, \mathcal{A})$ is \emph{relatively K\"ahler} if the cohomology class $\mathcal{A}$ is relatively K\"ahler, i.e. there is a K\"ahler form $\beta$ on $\mathbb{P}^1$ such that $\mathcal{A} + \pi^*\beta$ is K\"ahler on $\mathcal{X}$. 
\end{mydef}

\noindent Exploiting the discussion superceding Definition \ref{test config definition}, we \emph{in practice} restrict attention to the situation when $(\mathcal{X}, \mathcal{A})$ is a smooth (cohomological) test configuration for $(X,\alpha)$ dominating $X \times \mathbb{P}^1$, with $\mu: \mathcal{X} \rightarrow X \times \mathbb{P}^1$ the corresponding $\mathbb{C}^*$-equivariant bimeromorphic morphism. 
This situation is studied in detail in Section \ref{test config divisor section}, where we in particular show that the  class $\mathcal{A} \in H^{1,1}(\mathcal{X}, \mathbb{R})$ is always of the form $\mathcal{A} = \mu^*p_1^*\alpha + [D]$ for a unique $\mathbb{R}$-divisor $D$ supported on the central fiber, cf. Proposition \ref{divisor representation}. 


It is further natural to ask how the above notion of cohomological test configurations compares to the algebraic test configurations introduced by Donaldson in \cite{Donaldsontoric}. On the one hand, we have the following example:
\begin{ex} If $(\mathcal{Y}, \mathcal{L})$ is an algebraic test configuration for $(X,L)$ and we let $\bar{\mathcal{Y}}$, $\bar{\mathcal{L}}$ and $\bar{L}$ respectively denote the $\mathbb{C}^*$-equivariant compactifications over $\mathbb{P}^1$, then $(\bar{\mathcal{Y}}, c_1(\mathcal{L}))$ is a cohomological test configuration for $(X, c_1(L))$, canonically induced by $(\mathcal{Y}, \mathcal{L})$.
\end{ex}
 
On the other hand, there is no converse such correspondence; For instance, even if $(X,L)$ is a polarised manifold there are more cohomological test configurations $(\mathcal{X},\mathcal{A})$ for $(X,c_1(L))$ than algebraic test configurations $(\mathcal{Y}, \mathcal{L})$ for $(X,L)$. 
However, we show in Proposition \ref{prop comparison integral class} that such considerations are not an issue in the study of K-semistability of $(X,\alpha)$. 


\subsection{The Donaldson-Futaki invariant and K-semistability} \label{DF invariant and K-stability}
The following generalisation of the Donaldson-Futaki invariant is straightforward.
\begin{mydef} \label{DF definition}
Let $(\mathcal{X}, \mathcal{A})$ be a cohomological test configuration for $(X,\alpha)$. The Donaldson-Futaki invariant   of $(\mathcal{X}, \mathcal{A})$ is 
\begin{equation*}
\mathrm{DF}(\mathcal{X},\mathcal{A}) := \frac{\bar{\mathcal{S}}}{n+1}V^{-1} (\mathcal{A}^{n+1})_{\mathcal{X}} + V^{-1}(K_{\mathcal{X}/\mathbb{P}^1} \cdot \mathcal{A}^n)_{\mathcal{X}}.
\end{equation*}
\end{mydef}

\noindent We recall that $\mathcal{X}$ is assumed to be \emph{compact}, cf. Definition \ref{test config definition}, and that $K_{\mathcal{X}/\mathbb{P}^1} := K_{\mathcal{X}} - \pi^*K_{\mathbb{P}^1}$ denotes the relative canonical divisor.
The point is that by results of Wang \cite{Wang} and Odaka \cite{Odaka} $\mathrm{DF}(\bar{\mathcal{Y}},c_1(\mathcal{L}))$ coincides with $\mathrm{DF}(\mathcal{Y},\mathcal{L})$ whenever $(\mathcal{Y},\mathcal{L})$ is an algebraic test configuration for a polarised manifold $(X,L)$, with $\mathcal{Y}$ normal (see the proof of Proposition \ref{prop comparison integral class}). Hence the above quantity is a generalisation of the classical Donaldson-Futaki invariant. 

The analogue of K-semistability in the context of cohomological test configurations is defined as follows. 
\begin{mydef} \label{Ksemistable definition}
We say that $(X,\alpha)$ is \emph{ K-semistable} if $\mathrm{DF}(\mathcal{X},\mathcal{A}) \geq 0$ for all  relatively K\"ahler test configurations $(\mathcal{X},\mathcal{A})$ for $(X,\alpha)$. 
\end{mydef} 

\begin{rem} \emph{With the study of K-semistability in mind, we emphasise that the Donaldson-Futaki invariant $\mathrm{DF}(\mathcal{Y},\mathcal{L})$ (cf. \cite{Wang, Odaka}) depends only on $\mathcal{Y}$ and $c_1(\mathcal{L})$. The notion of cohomological test configuration emphasises this fact. }
\end{rem}

\noindent In order to further motivate the above definitions, we now introduce a number of related concepts and basic properties that will be useful in the sequel. 

\medskip

\subsection{Test configurations characterised by $\mathbb{R}$-divisors} \label{test config divisor section}

 Recall that if $(\mathcal{X},\mathcal{L})$ is an algebraic test configuration for a polarised manifold $(X,L)$ that dominates $(X,L) \times \mathbb{C}$, then $\mathcal{L} = \mu^* p_1^*L + D$ for a unique $\mathbb{Q}$-Cartier divisor $D$ supported on $\mathcal{X}_0$, see \cite{BHJ1}. Similarily, the following result characterises the classes $\mathcal{A}$ associated to smooth and dominating cohomological test configurations, in terms of $\mathbb{R}$-divisors $D$ supported on the central fiber $\mathcal{X}_0$. 

\begin{prop} \label{divisor representation}
Let $(\mathcal{X}, \mathcal{A})$ be a smooth cohomological test configuration for $(X, \alpha)$ dominating $X \times \mathbb{P}^1$, with $\mu: \mathcal{X} \rightarrow X \times \mathbb{P}^1$  the corresponding canonical $\mathbb{C}^*$-equivariant bimeromorphic morphism. Then there exists a unique $\mathbb{R}$-divisor $D$ supported on the central fiber $\mathcal{X}_0$ such that
\begin{equation*}
\mathcal{A} = \mu^*p_1^*\alpha + [D]
\end{equation*}
in $H^{1,1}(\mathcal{X}, \mathbb{R})$.
\end{prop}
\begin{proof}   
Let $\alpha := [\omega] \in H^{1,1}(X,\mathbb{R})$. We begin by proving existence: By hypothesis $\mathcal{X}$ dominates $X \times \mathbb{P}^1$ via the morphism $\mu$, such that the central fiber decomposes into the strict transform of $X \times \{0\}$ and the $\mu$-exceptional divisor. We write $\mathcal{X}_0 = \sum_i b_i E_i$,  with $E_i$ irreducible. Denoting by $[E]$ the cohomology class of $E$ and by $p_1: X \times \mathbb{P}^1 \rightarrow X$ the projection on the first factor, we then have the following formula:
\begin{lem}
$H^{1,1}(\mathcal{X}) = \mu^*p_1^*H^{1,1}(X) \; \oplus \; \bigoplus_i \mathbb{R} [E_i]$.
\end{lem}
\begin{proof} 
Let $\Theta$ be a closed $(1,1)$-form on $\mathcal{X}$. Then $T := \Theta -  \mu^*(\mu_*\Theta)$ is a closed $(1,1)$-current of order $0$ supported on $\cup_i E_i = \mathrm{Exc}(\mu)$. By Demailly's second theorem of support (see \cite{Demaillybook}) it follows that $T = \sum_i \lambda_i \delta_{E_j}$ and hence $[\Theta] = \mu^*(\mu_*[\Theta]) + \sum_i \lambda_i [E_i]$ in $H^{1,1}(\mathcal{X})$. 

Since $H^{1,1}(\mathbb{P}^1)$ is generated by $[0]$, we have $p_2^*H^{1,1}(\mathbb{P}^1) = \mathbb{R}[X \times \{0\}]$. By the K\"unneth formula, it thus follows that
$H^{1,1}(\mathcal{X}) = \mu^*p_1^*H^{1,1}(X) \; \oplus \; \mu^*(\mathbb{R}[X \times \{0\}])\; \oplus \; \bigoplus_i \mathbb{R} [E_i]$. 
\end{proof}
If we decompose $\mathcal{A}$ accordingly we obtain $\mathcal{A} = \mu^*p_1^*\eta + [D]$ with $D := \mu^*(c[X \times \{0\}]) + \bigoplus_i b_i [E_i]$ and $\eta$ a class in $H^{1,1}(X)$. The restrictions of $\mathcal{A}$ and $\mu^*p_1^*\alpha$ to $\pi^{-1}(1) \simeq X \times \{1\} \simeq X$ are identified with with $\alpha$ and $\eta$ respectively. Since $D$ is supported on $\mathcal{X}_0$ it follows that $\eta = \alpha$. We thus have the sought decomposition, proving existence. 

As for the uniqueness, we let $\mathcal{D}_0$ be the set of of $\mathbb{R}$-divisors $D$ with support contained in the central fiber $\mathcal{X}_0$. Consider the linear map
\begin{equation*}
R: \mathcal{D}_0 \rightarrow H^{1,1}(\mathcal{X})
\end{equation*}
\begin{equation*}
D \mapsto [D]
\end{equation*}
The desired uniqueness property is equivalent to injectivity of $R$. To this end, assume that $[D] = 0$ in $H^{1,1}(\mathcal{X})$. In particular $D_{\vert E_i} \equiv 0$ and it follows from a corollary of Zariski's lemma (see e.g. \cite[Lemma 8.2]{BHPV}) that $D = c\mathcal{X}_0$, with $c \in \mathbb{R}$. But, letting $\beta$ be any K\"ahler form on $X$, we see from the projection formula that
\begin{equation*}
(\mathcal{X}_0 \cdot (\mu^*p_1^*\beta)^n)_{\mathcal{X}} = ((X \times \{0\}) \cdot (p_1^*\beta)^n)_{X \times \mathbb{P}^1}  = \beta^n = V > 0,
\end{equation*} 
since $V$ is the K\"ahler volume. Hence  $[\mathcal{X}_0]$ is a non-zero class in $ H^{1,1}(\mathcal{X})$. It follows that $c = 0$, thus $D = 0$ as well. We are done.
\end{proof} 

\noindent This gives a very convenient characterisation of smooth cohomological test configurations for $(X,\alpha)$ that dominate $X \times \mathbb{P}^1$. 

\noindent In what follows, we will make use of resolution of singularities to associate a new test configuration $(\mathcal{X}', \mathcal{A}')$ for $(X,\alpha)$ to a given one, noting that this can be done \emph{without changing the Donaldson-Futaki invariant}. Indeed, by Hironaka \cite[Theorem 45]{Kollar} (see also Section \ref{test config definition}) there is a $\mathbb{C}^*$-equivariant proper bimeromorphic map $\mu: \mathcal{X}' \rightarrow \mathcal{X}$, with $\mathcal{X}'$ smooth and such that $\mathcal{X}_0'$ has simple normal crossings. Moreover, $\mu$ is an isomorphism outside of the central fiber $\mathcal{X}_0$. Set $\mathcal{A}' := \mu^*\mathcal{A}$. By the projection formula we then have 
\begin{equation*}
\mathrm{DF}(\mathcal{X}', \mathcal{A}') = \frac{\bar{\mathcal{S}}}{n+1}V^{-1} ((\mathcal{A}')^{n+1})_{\mathcal{X}'} + V^{-1}(K_{\mathcal{X}'/\mathbb{P}^1} \cdot (\mathcal{A}')^n)_{\mathcal{X}'}   
\end{equation*} 
\begin{equation*}
 = \frac{\bar{\mathcal{S}}}{n+1}V^{-1}(\mathcal{A}^{n+1})_{\mathcal{X}} + V^{-1}(K_{\mathcal{X}/\mathbb{P}^1} \cdot \mathcal{A}^n)_{\mathcal{X}} = \mathrm{DF}(\mathcal{X}, \mathcal{A}).
\end{equation*} 

The following result states that it suffices to test K-semistability for a certain class of cohomological test configurations 'characterised by an $\mathbb{R}$-divisor' (in the above sense of Proposition \ref{divisor representation}).  

\begin{prop} \label{lemma good divisor}
Let $\alpha \in H^{1,1}(X,\mathbb{R})$ be K\"ahler. Then $(X,\alpha)$ is K-semistable if and only if $\mathrm{DF}(\mathcal{X},\mathcal{A}) \geq 0$ for all smooth, relatively K\"ahler cohomological test configurations $(\mathcal{X},\mathcal{A})$ for $(X, \alpha)$ dominating $X \times \mathbb{P}^1$. 
\end{prop}
\begin{proof}
 Let $(\mathcal{X},\mathcal{A})$ be any cohomological test configuration for $(X, \alpha)$ that is relatively K\"ahler. By Hironaka (see \cite{Kollar}) there is a sequence of blow ups $\rho: \mathcal{X}' \rightarrow X \times \mathbb{P}^1$ with smooth $\mathbb{C}^*$-equivariant centers such that $\mathcal{X}'$ simultaneously dominates $\mathcal{X}$ and $X \times \mathbb{P}^1$ via morphisms $\mu$ and $\rho$ respectively. Moreover, there is a divisor $E$ on $\mathcal{X}'$ that is $\rho$-exceptional and $\rho$-ample (and antieffective, i.e. $-E$ is effective). By Proposition \ref{divisor representation}, we have
\begin{equation*}
\mu^*\mathcal{A} =  \rho^*p_1^*\alpha + [D], 
\end{equation*}
where $D$ is an $\mathbb{R}$-divisor on $\mathcal{X}'$ supported on $\mathcal{X}'_0$. Note that the class $\mu^*\mathcal{A} \in H^{1,1}(\mathcal{X}', \mathbb{R})$ is relatively nef. 

We proceed by perturbation; Since $\alpha$ is K\"ahler on $X$, we may pick a K\"ahler class $\eta$ on $\mathbb{P}^1$ such that $p_1^*\alpha + p_2^*\eta =: \beta$ is K\"ahler on $X \times \mathbb{P}^1$. Since $E$ is $\rho$-ample one may in turn fix an $\varepsilon \in (0,1)$ sufficiently small such that $\rho^*\beta + \varepsilon[E]$ is K\"ahler on $\mathcal{X}'$. It follows that $\rho^*p_1^*\alpha + \varepsilon [E]$
is \emph{relatively} K\"ahler (with respect to $\mathbb{P}^1$) on $\mathcal{X}'$. Thus $\rho^*p_1^*\alpha + [D] + \delta(\rho^*p_1^*\alpha + \varepsilon[E]) $ is relatively K\"ahler for all $\delta \geq 0$ small enough. In turn, so is $\mathcal{A}'_{\delta} := \rho^*p_1^*\alpha + [D_{\delta}]$,
where $D_{\delta}$ denotes the convex combination $D_{\delta} := \frac{1}{1+\delta}D + \frac{\delta \varepsilon }{1 + \delta}E$.
Assuming that the $\mathrm{DF}$-invariant of a smooth and dominating test configuration is always non-negative, it follows from the projection formula and continuity of the Donaldson-Futaki invariant, that
\begin{equation*}
0 \leq \mathrm{DF}(\mathcal{X}',\mathcal{A}'_{\delta}) \longrightarrow \mathrm{DF}(\mathcal{X}',\mu^*\mathcal{A}) = \mathrm{DF}(\mathcal{X},\mathcal{A}). 
\end{equation*}
as $\delta \rightarrow 0$. The other direction holds by definition, so this proves the first part of the lemma. 

\end{proof}

\begin{rem} With respect to testing K-semistability one can in fact restrict the class of test configurations that need to be considered even further, as explained in Section \ref{further reduction techniques}.
\end{rem} 


\subsection{Cohomological K-semistability for polarised manifolds} \label{K-semistability for polarised manifolds}

It is useful to compare cohomological- and algebraic K-semistability in the special case of a polarised manifold $(X,L)$.

\begin{prop} \label{prop comparison integral class}
Let $(X,L)$ be a polarised manifold and let $\alpha:= c_1(L)$. Then $(X,c_1(L))$ is \emph{(cohomologically)} K-semistable if and only if $(X,L)$ is  \emph{(algebraically)} K-semistable. 
\end{prop}
\begin{proof}
Suppose that $(X, c_1(L))$ is cohomologically K-semistable. If $(\mathcal{X}, \mathcal{L})$ is an ample test configuration for $(X,L)$, let $\mathcal{A} := c_1(\bar{\mathcal{L}})$. By the intersection theoretic characterisation of the Donaldson-Futaki invariant (Definition \ref{DF definition}) we then have $\mathrm{DF}(\mathcal{X},\mathcal{A}) = \mathrm{DF}(\mathcal{X},\mathcal{L}) \geq 0$. Hence $(X,L)$ is algebraically K-semistable.

Conversely, suppose that $(X,L)$ is algebraically K-semistable and let $(\mathcal{X}, \mathcal{A})$ be a cohomological test configuration for $(X,\alpha)$.  By Lemma \ref{lemma good divisor} we may assume that $(\mathcal{X}, \mathcal{A})$ is a smooth, relatively K\"ahler test configuration for $(X,\alpha)$ dominating $X \times \mathbb{P}^1$, with $\mu:  \rightarrow X \times \mathbb{P}^1$ the corresponding $\mathbb{C}^*$-equivariant bimeromorphic morphism. By Proposition \ref{divisor representation} we further have $\mathcal{A} = \mu^*p_1^*c_1(L) + [D]$ for a uniquely determined $\mathbb{R}$-divisor $D$ on $\mathcal{X}$ supported on the central fiber $\mathcal{X}_0$. Since $\mathcal{A}$ is relatively K\"ahler, there is a K\"ahler form $\eta$ on $\mathbb{P}^1$ such that $\mathcal{A} + \pi^*\eta$ is K\"ahler on $\mathcal{X}$. Approximating the coefficients of the divisor $D$ by a sequence of rationals, we write $D = \lim D_j$ for $\mathbb{Q}$-divisors $D_j$ on $\mathcal{X}$, all supported on $\mathcal{X}_0$. As $j \rightarrow +\infty$, we then have
\begin{equation*}
\mu^*p_1^*c_1(L) + [D_j] + \pi^*\eta \longrightarrow \mathcal{A} + \pi^*\eta,
\end{equation*}
which is a K\"ahler form on $\mathcal{X}$. Since the K\"ahler cone is open, it follows that $\mu^*p_1^*c_1(L) + [D_j] + \pi^*\eta$ is also K\"ahler for all $j$ large enough. 

Now let $\mathcal{L}_j := \mu^*p_1^*L + D_j$. By the above, $\mathcal{L}_j$ is a relatively ample $\mathbb{Q}$-line bundle over $\mathcal{X}$ and $c_1(\mathcal{L}_j) \rightarrow \mathcal{A}$. We thus conclude that $(\mathcal{X}, \mathcal{L}_j)$ (for all $j$ large enough) is an ample test configuration for $(X,L)$. Hence
\begin{equation*}
0 \leq \mathrm{DF}(\mathcal{X}, \mathcal{L}_j) \longrightarrow \mathrm{DF}(\mathcal{X}, \mathcal{A}). 
\end{equation*}
as $j \rightarrow +\infty$, which is what we wanted to prove.
\end{proof}


\subsection{The non-Archimedean Mabuchi functional and base change} \label{further reduction techniques} 
Let $(\mathcal{X},\mathcal{A})$ be a cohomological test configuration for $(X,\alpha)$. A natural operation on $(\mathcal{X}, \mathcal{A})$ is that of \emph{base change} (on $\mathcal{X}$ and we pull back $\mathcal{A}$). Unlike resolution of singularities, however, the $\mathrm{DF}$-invariant does not behave well under under base change. In this context, a more natural object of study is instead the \emph{non-Archimedean Mabuchi functional} $\mathrm{M}^{\mathrm{NA}}$ (first introduced in \cite{BHJ2} and \cite{BHJ1}, where also an explanation of the terminology is given).

\begin{mydef} The \emph{non-Archimedean Mabuchi functional} is the modification of the Donaldson-Futaki invariant given by 
\begin{equation*}
\mathrm{M}^{\mathrm{NA}}(\mathcal{X}, \mathcal{A}) :=  \mathrm{DF}(\mathcal{X}, \mathcal{A}) + V^{-1}((\mathcal{X}_{0,\mathrm {red}} - \mathcal{X}_0) \cdot \mathcal{A}^n)_{\mathcal{X}}.
\end{equation*}
\end{mydef} 

\noindent Note that the 'correction term' $V^{-1}((\mathcal{X}_{0,\mathrm {red}} - \mathcal{X}_0) \cdot \mathcal{A}^n)_{\mathcal{X}}$ is non-positive and vanishes precisely when the central fiber $\mathcal{X}_0$ is reduced. 
The point of adding to $\mathrm{DF}$ this additional term is that the resulting quantity $\mathrm{M}^{\mathrm{NA}}(\mathcal{X}, \mathcal{A})$ becomes \emph{homogeneous under base change}, i.e. we have the following lemma.

\begin{lem} \label{homogeneity of mabuchi} \emph{(\cite{BHJ1})} Let $(\mathcal{X},\mathcal{A})$ be a cohomological test configuration for $(X,\alpha)$ and let $d \in \mathbb{N}$. Denote by $\mathcal{X}_d$ the normalisation of the base change of $\mathcal{X}$, by $g_d: \mathcal{X}_d \rightarrow \mathcal{X}$ the corresponding morphism \emph{(of degree $d$)} and set $\mathcal{A}_d := g_d^*\mathcal{A}$. Then
\begin{equation*}
\mathrm{M}^{\mathrm{NA}}(\mathcal{X}_d, \mathcal{A}_d) = d \cdot \mathrm{M}^{\mathrm{NA}}(\mathcal{X}, \mathcal{A}).
\end{equation*}
\end{lem}
\begin{proof}
We refer the reader to \cite[Proposition 7.13]{BHJ1}, whose proof goes through in the analytic case as well. 
\end{proof}

As an application, it follows from Mumford's semistable reduction theorem (\cite[p.53]{KKMS}, see also \cite[\S 16, p.6]{KNX} for a remark on the analytic case) that there is a $d \in \mathbb{N}$, a finite base change $f:\tau \mapsto \tau^d$ (for $d$ 'divisible enough'), a smooth test configuration $\mathcal{X}'$ and a diagram 

\[\begin{tikzcd} \mathcal{X} \arrow[swap]{d}{\pi} & \arrow[swap]{l}{g_d} \mathcal{X}_d \arrow{d}{\pi_d} & \arrow[swap]{l}{\rho} \arrow{dl}{\pi'} \mathcal{X}' \\ \mathbb{P}^1 & \arrow{l}{f} \mathbb{P}^1 \end{tikzcd} \]

\noindent such that $\mathcal{X}'$ is semistable, i.e. smooth and such that $\mathcal{X}_0'$ is a reduced divisor with simple normal crossings. In particular, note that the correction term $V^{-1}((\mathcal{X}_{0,\mathrm{red}}' - \mathcal{X}_0') \cdot \mathcal{A}^n)_{\mathcal{X}'}$ vanishes. Here $\mathcal{X}_d$ denotes the normalisation of the base change, which is dominated by the semistable test configuration $\mathcal{X}'$ for $X$. Moreover, $g_d \circ \rho$ is an isomorphism over $\mathbb{P}^1 \setminus \{0\}$. 

Letting $\mathcal{A}_d := g_d^*\mathcal{A}$ be the pullback of $\mathcal{A}$ to $\mathcal{X}_d$, and $\mathcal{A}' := \rho^* \mathcal{A}_d$ the pullback to $\mathcal{X}'$, it follows from the above homogeneity of the $\mathrm{M}^{\mathrm{NA}}$ that
\begin{equation*}
\mathrm{DF}(\mathcal{X}', \mathcal{A}') = \mathrm{M}^{\mathrm{NA}}(\mathcal{X}', \mathcal{A}') =  \mathrm{M}^{\mathrm{NA}}(\mathcal{X}_d, \mathcal{A}_d) = d \cdot \mathrm{M}^{\mathrm{NA}}(\mathcal{X}, \mathcal{A}) \leq d \cdot \mathrm{DF}(\mathcal{X}, \mathcal{A}),
\end{equation*}
where $d$ is the degree of $g_d$. We have thus associated to $(\mathcal{X}, \mathcal{A})$ a new test configuration $(\mathcal{X}', \mathcal{A}')$ for $(X,\alpha)$ such that the total space $\mathcal{X}'$ is semistable. Up to replacing $\mathcal{X}'$ with a \emph{determination} (see Section \ref{compactification}) we can moreover assume that $\mathcal{X}'$ dominates $X \times \mathbb{P}^1$. Hence, the above shows that  $\mathrm{DF}(\mathcal{X}, \mathcal{A}) \geq \mathrm{DF}(\mathcal{X}', \mathcal{A}')/d$. By an argument by perturbation much as the one in the proof of Proposition \ref{lemma good divisor}, we obtain the following stronger version of the aforementioned result.

\begin{prop} \label{further simplification} Let $\alpha \in H^{1,1}(X,\mathbb{R})$ be K\"ahler. Then $(X,\alpha)$ is K-semistable \emph{(Definition \ref{Ksemistable definition})} if and only if  $\mathrm{DF}(\mathcal{X}, \mathcal{A}) \geq 0$ for all semistable, relatively K\"ahler cohomological test configurations $(\mathcal{X}, \mathcal{A})$ for $(X, \alpha)$ dominating $X \times \mathbb{P}^1$.
\end{prop}

\medskip

\section{Transcendental Kempf-Ness type formulas} \label{proof B} 
Let $X$ be a compact K\"ahler manifold of dimension $n$ and let $\theta_i$, $0 \leq i \leq n$, be closed $(1,1)$-forms on $X$. Let $\alpha_i := [\theta_i] \in H^{1,1}(X,\mathbb{R})$ be the corresponding cohomology classes. 
In this section we aim to prove Theorem B. In other words, we establish a Kempf-Ness type formula (for \emph{cohomological} test configurations), which connects the asymptotic slope of the multivariate energy functional $\langle \varphi_0^t,\dots,\varphi_n^t \rangle_{(\theta_0, \dots, \theta_n)}$ (see Definition \ref{multivariate energy functional}) with a certain intersection number. In order for such a result to hold, we need to ask that the rays $(\varphi_i^t)_{t \geq 0}$ are \emph{compatible} with $(\mathcal{X}_i, \mathcal{A}_i)$ in a sense that has to do with extension across the central fiber, see Section \ref{smoothcompatibility} below. 


\noindent For what follows, note that, by equivariant resolution of singularities, there is a test configuration $\mathcal{X}$ for $X$ which is smooth and dominates $X \times \mathbb{P}^1$. This setup comes with \emph{canonical}  $\mathbb{C}^*$-equivariant bimeromorphic maps $\rho_i: \mathcal{X} \rightarrow \mathcal{X}_i$ respectively. In particular:
\begin{mydef} \label{definition intersection number} 
We define the intersection number
\begin{equation*} 
(\mathcal{A}_0 \cdot \dots \cdot \mathcal{A}_n) := (\rho_0^* \mathcal{A}_0 \cdot \dots \cdot \rho_n^* \mathcal{A}_n)_{\mathcal{X}}
\end{equation*}
by means of pulling back the respective cohomology classes to $\mathcal{X}$. 
\end{mydef}

\begin{rem} Up to desingularising we can and we will in this section consider only smooth cohomological test configurations $(\mathcal{X}_i,\mathcal{A}_i)$ for $(X, \alpha_i)$ dominating $X \times \mathbb{P}^1$, with $\mu_i: \mathcal{X}_i \rightarrow X \times \mathbb{P}^1$ the corresponding $\mathbb{C}^*$-equivariant bimeromorphic morphisms respectively. We content ourselves by noting that the following $\mathcal{C}^{\infty}$-compatibility condition can be defined (much in the same way, using a desingularisation) in the singular case as well. 

\end{rem}

 

\subsection{Compatibility of rays and test configurations} \label{smoothcompatibility}
Let $(\mathcal{X},\mathcal{A})$ be a smooth (cohomological) test configuration for $(X,\alpha)$ dominating $X \times \mathbb{P}^1$, with $\mu:\mathcal{X} \rightarrow X \times \mathbb{P}^1$ the corresponding canonical $\mathbb{C}^*$-equivariant bimeromorphic morphism. We then have 
$$
\mathcal{A}=\mu^*p_1^*\alpha+[D]
$$ 
for a unique $\mathbb{R}$-divisor $D$ supported on $\mathcal{X}_0$, with $p_1:X \times \mathbb{P}^1 \rightarrow X$ denoting the first projection, cf. Proposition \ref{divisor representation}.

We fix the choice of an $S^1$-invariant function 'Green function' $\psi_D$ for $D$, so that $\delta_D=\theta_D+dd^c\psi_D$, with $\theta_D$ a smooth $S^1$-invariant closed $(1,1)$-form on $\mathcal{X}$. Locally, we thus have 
$$
\psi_D = \sum_j a_j \log|f_j| \; \; \textrm{mod} \; \mathcal{C}^{\infty}, 
$$ 
where (writing $D := \sum_j a_jD_j$ for the decomposition of $D$ into irreducible components) the $f_j$ are local defining equations for the $D_j$ respectively. 
In particular, the choice of $\psi_D$ is uniquely determined modulo a smooth function. 

The main purpose of this section is to establish Theorem B, which is a formula relating algebraic (intersection theoretic) quantities to asymptotic slopes of Deligne functionals (e.g. $\mathrm{E}$ or $\mathrm{J}$) along certain rays. However, such a formula can not hold for \emph{any} such ray. The point of the following \emph{compatibiltiy conditions} is to establish some natural situations in which this formula holds. 
Technically, recall that a ray $(\varphi_t)_{t \geq 0}$ on $X$ is in correspondence with an $S^1$-invariant functions $\Phi$ on $X \times \bar{\Delta}^*$. The proof of Theorem B below, will show that it is important to extend the function $\Phi \circ \mu$ on $\mathcal{X} \setminus \mathcal{X}_0$ also across the central fiber $\mathcal{X}_0$. 

To this end, we introduce the notions of $\mathcal{C}^{\infty}$-, $L^{\infty}$- and $\mathcal{C}^{1,1}_{\mathbb{C}}$-compatibility between the ray $(\varphi_t)_{t \geq 0}$ and the test configuration $(\mathcal{X}, \mathcal{A})$. The purpose of introducing more than one version of compatibility is that we will distuingish between the following two situations of interest to us.
\begin{enumerate}
\item smooth but not necessarily subgeodesic rays $(\varphi_t)$ that are $\mathcal{C}^{\infty}$-compatible with the smooth test configuration $(\mathcal{X}, \mathcal{A})$ for $(X,\alpha)$, dominating $X \times \mathbb{P}^1$. Here we can consider $\alpha = [\theta] \in H^{1,1}(X,\mathbb{R})$ for any closed $(1,1)$-form $\theta$ on $X$. 
\item locally bounded \emph{subgeodesic} rays $(\varphi_t)$ that are $L^{\infty}$-compatible or (more restrictively) $\mathcal{C}^{1,1}_{\mathbb{C}}$-compatible with the given smooth and \emph{relatively K\"ahler} test configuration $(\mathcal{X}, \mathcal{A})$ for $(X,\alpha)$, dominating $X \times \mathbb{P}^1$. Here we thus suppose that $\alpha$ is a K\"ahler class. 
\end{enumerate} 

\noindent Theorem B has valid formulations in both these situations, as pointed out in Remark \ref{Remark Theorem B}. The second situation is interesting notably with weak geodesic rays in mind, cf. Section \ref{section weak geodesic}. 


\subsection{$\mathcal{C}^{\infty}$-compatible rays}
We first introduce the notion of smooth (not necessarily subgeodesic) rays that are \emph{$\mathcal{C}^{\infty}$-compatible} with the given test configuration $(\mathcal{X}, \mathcal{A})$ for $(X,\alpha)$.

\begin{mydef} 
Let $(\varphi_t)_{t \geq 0}$ be a smooth ray in $\mathcal{C}^{\infty}(X)$, and denote by $\Phi$ the corresponding smooth $S^1$-invariant function on $X \times \bar{\Delta}^*$. We say that $(\varphi_t)$ and $(\mathcal{X},\mathcal{A})$ are \emph{$\mathcal{C}^{\infty}$-compatible} if $\Phi \circ \mu+\psi_D$ extends smoothly across $\mathcal{X}_0$. 
\end{mydef}
The condition is indeed independent of the choice of $\psi_D$, as the latter is well-defined modulo a smooth function. In the case of a polarised manifold $(X,L)$ with an (algebraic) test configuration $(\mathcal{X}, \mathcal{L})$ this condition amounts to demanding that the metric on $\mathcal{L}$ associated to the ray $(\varphi_t)_{t \geq 0}$ extends  smoothly across the central fiber.

As a useful 'model example' to keep in mind, let $\Omega$ be a smooth $S^1$-invariant representative of $\mathcal{A}$ and denote the restrictions  $\Omega_{\vert \mathcal{X}_{\tau}} =: \Omega_{\tau}$. Note that $\Omega_{\tau}$ and $\Omega_1$ are cohomologous for each $\tau \in \mathbb{P}^1 \setminus \{0\}$, and hence we may define a ray $(\varphi_t)_{t \geq 0}$ on $X$, $\mathcal{C}^{\infty}$-compatible with $(\mathcal{X}, \mathcal{A})$, by the following relation
$\lambda(\tau)^*\Omega_{\tau} - \Omega_1 = dd^c\varphi_{\tau}$,
where $t = -\log|\tau|$ and $\lambda(\tau): \mathcal{X}_{\tau} \rightarrow \mathcal{X}_1 \simeq X$ is the isomorphism induced by the $\mathbb{C}^*$-action $\lambda$ on $\mathcal{X}$.


We further establish  existence of a smooth $\mathcal{C}^{\infty}$-compatible \emph{subgeodesic} ray associated to a given relatively K\"ahler test configuration $(\mathcal{X}, \mathcal{A})$ for $(X,\alpha)$. 

\begin{lem}\label{lem:smooth} If $\mathcal{A}$ is relatively K\"ahler, then $(\mathcal{X},\mathcal{A})$ is $\mathcal{C}^{\infty}$-compatible with some smooth subgeodesic ray $(\varphi_t)$.
\end{lem}
\begin{proof} Since $\mathcal{A}$ is relatively K\"ahler, it admits a smooth $S^1$-invariant representative $\Omega$ with $\Omega+\pi^*\eta>0$ for some $S^1$-invariant K\"ahler form $\eta$ on $\mathbb{P}^1$. By the $dd^c$-lemma on $\mathcal{X}$, we have $\Omega=\mu^*p_1^*\omega+\theta_D+dd^c u$ for some $S^1$-invariant $u\in \mathcal{C}^{\infty}(X)$, which may be assumed to be $0$ after replacing $\psi_D$ with $\psi_D-u$. As a result, we get
$$
\Omega=\mu^*p_1^*\omega+\delta_D-dd^c\psi_D.
$$
We may also choose a smooth $S^1$-invariant function $f$ on a neighborhood $U$ of $\bar{\Delta}$ such that $\eta_{\vert U}=dd^c f$, and a constant $A \gg 1$ such that $D\leq A\mathcal{X}_0$. Using the Lelong-Poincar\'e formula $\delta_{\mathcal{X}_0}=dd^c\log|\tau|$ we get
$$
0<\Omega+\pi^*\eta=\mu^*p_1^*\omega+\delta_{D-A\mathcal{X}_0}+dd^c\left(f\circ\pi+A\log|\tau|-\psi_D\right)
$$
on $\pi^{-1}(U)$. Since $D-A\mathcal{X}_0\leq 0$, it follows that $f\circ\pi+A\log|\tau|-\psi_D$ is $\mu^*p_1^*\omega$-psh, and hence descends to an $S^1$-invariant $p_1^*\omega$-psh function $\tilde{\Phi}$ on $X\times U$ (because the fibers of $\mu$ are compact and connected, by Zariski's main theorem). The ray associated with the $S^1$-invariant function $\Phi:=\tilde{\Phi}-A\log|\tau|$ has the desired properties. 
\end{proof}

%
%


\subsection{$\mathcal{C}^{1,1}_{\mathbb{C}}$-compatible rays and the weak geodesic ray associated to $(\mathcal{X}, \mathcal{A})$} \label{section weak geodesic}
Let $(\mathcal{X}, \mathcal{A})$ be a smooth, relatively K\"ahler cohomological test configuration for $(X,\alpha)$ (with $\alpha$ K\"ahler). With this setup, it is also interesting to consider the following weaker compatibility conditions, referred to as $L^{\infty}$-compatibility and  $\mathcal{C}^{1,1}_{\mathbb{C}}$-compatibility respectively.  

\begin{mydef} \label{weak compatibility} Let $(\varphi_t)_{t \geq 0}$ be a locally bounded subgeodesic ray, and denote by $\Phi$ the corresponding $S^1$-invariant locally bounded $p_1^*\omega$-psh function on $X\times\bar{\Delta}^*$. We say that $(\varphi_t)$ and $(\mathcal{X},\mathcal{A})$ are \emph{$L^{\infty}$-compatible} if $\Phi\circ\mu+\psi_D$ is locally bounded near $\mathcal{X}_0$, resp. \emph{$\mathcal{C}^{1,1}_{\mathbb{C}}$-compatible} if $\Phi\circ\mu+\psi_D$ is of class $\mathcal{C}^{1,1}_{\mathbb{C}}$ on $\pi^{-1}(\Delta)$. 
\end{mydef}

\noindent 
Indeed, we will see that $\mathcal{C}^{1,1}_{\mathbb{C}}$-compatibility is always satisfied for weak geodesic rays associated to $(\mathcal{X}, \mathcal{A})$. In particular, for any given test configuration, $\mathcal{C}^{1,1}_{\mathbb{C}}$-compatible subgeodesics always exist.
This is the content of the following result, which is a consequence of the theory for degenerate Monge-Amp\`ere equations on manifolds with boundary. We refer the reader to \cite{LNM} for the relevant background. 

\begin{lem} \label{weak geodesic construction} 
With the situation $(2)$ in mind, let $(\mathcal{X},\mathcal{A})$ be a smooth, relatively K\"ahler cohomological test configuration of $(X,\alpha)$ dominating $X \times \mathbb{P}^1$. Then $(\mathcal{X}, \mathcal{A})$ is $\mathcal{C}^{1,1}_{\mathbb{C}}$-compatible with some weak geodesic ray $(\varphi_t)_{t \geq 0}$. 
\end{lem}

\begin{rem} \emph{The proof will show that the constructed ray is actually unique, once a $\varphi_0 \in \mathcal{H}$ is fixed.}
\end{rem}

\begin{proof}[Proof of Lemma \ref{weak geodesic construction}]
Let $M := \pi^{-1}(\bar{\Delta}) \subset \mathcal{X} $. It is a smooth complex manifold with boundary $\partial M = \pi^{-1}(S^1)$. 

Let $D$, $\theta_D$, $\psi_D$ and $\Omega$ be as above. Since $\Omega$ is relatively K\"ahler there is an $\eta \in H^{1,1}(\mathbb{P}^1)$ such that $\Omega + \pi^*\eta$ is K\"ahler on $\mathcal{X}$. We may then write $\tilde{\Omega} = \Omega + \pi^*\eta + dd^cg$,
where $\tilde{\Omega}$ is a K\"ahler form on $\mathcal{X}$ and  $g \in \mathcal{C}^{\infty}(\mathcal{X})$. In a neighbourhood  of $\bar{\Delta}$ the form $\eta$ is further $dd^c$-exact, and so we write $\eta = dd^c (g' \circ \pi)$ for a smooth function $g' \circ \pi$ on $\bar{\Delta}$.  
We now consider the following degenerate complex Monge-Amp\`ere equation;  \medskip

\[
    (\star)  \left\{
                \begin{array}{ll}
                  (\tilde{\Omega} + dd^c\tilde{\Psi})^{n+1} = 0 \; \; \mathrm{on}\; \mathrm{Int}(M) \\
                  \tilde{\Psi}_{\vert \partial M} = \varphi_0 + \psi_D - g' -g 
                \end{array}
              \right.
  \] \medskip
  
\noindent Since $\tilde{\Omega}$ is K\"ahler, it follows that there exists a unique $\tilde{\Omega}$-psh function $\tilde{\Psi}$ solving $(\star)$ and that is moreover of class $\mathcal{C}^{1,1}_{\mathbb{C}}$ (see for instance \cite[Theorem B]{LNM}. We now define a $p_1^*\omega$-psh function on $X \times \bar{\Delta}^* \hookrightarrow \mathcal{X}$ by $\mu^*\Phi = \tilde{\Psi} - \psi_D + g' + g$. 
We then have 
$$
\mu^*(p_1^*\omega + dd^c\Phi) = \tilde{\Omega} + dd^c\tilde{\Psi}
$$
on $\pi^{-1}(\bar{\Delta}^*)$. In particular, $\Phi$ defines a weak geodesic ray $(\varphi_t)_{t \geq 0}$ on $X$.  Moreover, the current
$$
\mu^*dd^c\Phi + \delta_D = dd^c\tilde{\Psi} + \delta_D - dd^c\psi_D = dd^c\tilde{\Psi} + \theta_D
$$
has locally bounded coefficients. Indeed, $dd^c\tilde{\Psi} \in L^{\infty}_{\mathrm{loc}}$ (as solution of $(\star)$, cf. \cite{LNM}) and $\theta_D$ is a smooth $(1,1)$-form on $\bar{\mathcal{X}}$. The constructed ray is thus $\mathcal{C}^{1,1}_{\mathbb{C}}$-compatible with $(\mathcal{X}, \mathcal{A})$.
\end{proof}


\subsection{A useful lemma}
We now note that in order to compute the asymptotic slope of the Monge-Amp\`ere energy functional $\mathrm{E}$ or its multivariate analogue $\mathrm{E}_{(\omega_0, \dots, \omega_n)}$ we may in fact replace $L^{\infty}$-compatible rays $(\varphi^t)$ with $(\mathcal{X}, \mathcal{A})$ by $\mathcal{C}^{\infty}$-compatible ones. Indeed, note that any two locally bounded subgeodesic rays $(\varphi_t)$ and $(\varphi'_t)$ $L^{\infty}$-compatible with $(\mathcal{X},\mathcal{A})$ satisfy $\Phi\circ\mu=\Phi'\circ\mu+O(1)$ near $\mathcal{X}_0$, and hence $\varphi_t=\varphi'_t+O(1)$ as $t\rightarrow +\infty$. This leads to the following observation, which will be useful in view of proving Theorems B and C. 

\begin{lem} \label{lemma: replacing ray}
Let $(\mathcal{X}_i, \mathcal{A}_i)$ be smooth, relatively K\"ahler cohomological test configurations for $(X,\alpha_i)$ respectively, dominating $X \times \mathbb{P}^1$. Let $(\varphi_i^t)_{t \geq 0}$ and $({\varphi'}_i^t)_{t \geq 0}$ be locally bounded subgeodesics that are $L^{\infty}$-compatible with $(\mathcal{X}_i, \mathcal{A}_i)$ respectively.
Then
$$
\langle \varphi_0^t, \varphi_1^t, \dots, \varphi_n^t \rangle_{(\omega_0, \dots, \omega_n)} = \langle {\varphi'}_0^t, \varphi_1^t, \dots, \varphi_n^t \rangle_{(\omega_0, \dots, \omega_n)} + O(1)
$$
as $t \rightarrow +\infty$. 
\end{lem} 
\begin{proof}
For each $i$, $0 \leq i \leq n$, we have $\varphi_i^t={\varphi'}_i^t+O(1)$ as $t\rightarrow +\infty$. Recall that the mass of the Bedford-Taylor product $\bigwedge(\omega_i + dd^c\varphi_i^t)$ is computed in cohomology, thus independent of $t$. Hence, the quantity
$$
\langle \varphi_0^t, \varphi_1^t, \dots, \varphi_n^t \rangle_{(\omega_0, \dots, \omega_n)} - \langle {\varphi'}_0^t, \varphi_1^t, \dots, \varphi_n^t \rangle_{(\omega_0, \dots, \omega_n)}
$$
$$
 = \int_X (\varphi_0^t - {\varphi'}_0^t) (\omega_1 + dd^c\varphi_1^t) \wedge \dots \wedge (\omega_n + dd^c\varphi_n^t)
$$
is bounded as $t \rightarrow +\infty$. By symmetry, the argument may be repeated for the remaining $i$, yielding the result. 
\end{proof}


\subsection{Asymptotic slope of Deligne functionals. Proof of Theorem B}
With the above formalism in place, we are ready to formulate the main result of this section (Theorem B of the introduction). It constitutes the main contribution towards establishing Theorem A, and may be viewed as a transcendental analogue of Lemma 4.3 in \cite{BHJ2}. We here formulate and prove the theorem in the 'smooth but not necessarily K\"ahler' setting (see Section \ref{smoothcompatibility}, situation $(1)$). However, one should note that there is also a valid formulation for $L^{\infty}$-compatible subgeodesics, as pointed out in Remark \ref{Remark Theorem B}. 

\begin{thm} \label{Theorem C} 
Let $X$ be a compact K\"ahler manifold of dimension $n$ and let $\theta_i$, $0 \leq i \leq n$, be closed $(1,1)$-forms on $X$. Set $\alpha_i := [\theta_i] \in H^{1,1}(X,\mathbb{R})$. Consider smooth cohomological test configurations  $(\mathcal{X}_i, \mathcal{A}_i)$ for $(X,\alpha_i)$ dominating $X \times \mathbb{P}^1$. For each collection of smooth rays $(\varphi_i^t)_{t \geq 0}$ $\mathcal{C}^{\infty}$-compatible with $(\mathcal{X}_i, \mathcal{A}_i)$ respectively, the asymptotic slope of the multivariate energy functional $\langle \cdot, \dots, \cdot \rangle := \langle \cdot, \dots, \cdot \rangle_{(\theta_0, \dots, \theta_n)}$ is well-defined and satisfies
\begin{equation*}
\frac{\langle \varphi_0^t, \dots, \varphi_n^t \rangle}{t} \longrightarrow (\mathcal{A}_0 \cdot \dots \cdot \mathcal{A}_n)
\end{equation*}
as $t \rightarrow +\infty$. See \ref{definition intersection number} for the definition of the above intersection number in case the $\mathcal{X}_i$ are not all equal.  
\end{thm}
\begin{proof}
 Fix any smooth $S^1$-invariant $(1,1)$-forms $\Omega_i$  on $\mathcal{X}_i$ such that $[\Omega_i] =\mathcal{A}_i$ in $H^{1,1}(\mathcal{X}_i, \mathbb{R})$. Let $(\varphi_i^t)_{t\geq 0}$ be smooth and $\mathcal{C}^{\infty}$-compatible with $(\mathcal{X}_i,\mathcal{A}_i)$ respectively. Let $\mathcal{X}$ be a smooth test configuration that simultaneously dominates the $\mathcal{X}_i$. By pulling back to $\mathcal{X}$ we can assume that the $\mathcal{X}_i$ are all equal (note that the notion of being $\mathcal{C}^{\infty}$-compatible is preserved under this pull-back). 

In the notation of Section \ref{smoothcompatibility}, the functions $\Phi_i\circ\mu+\psi_D$ are then smooth on the manifold with boundary $M:=\pi^{-1}(\bar{\Delta})$, and may thus be written as the restriction of smooth $S^1$-invariant functions $\Psi_i$ on $\mathcal{X}$ respectively. 

Using the $\mathbb{C}^*$-equivariant isomorphism $\mathcal{X}\setminus\mathcal{X}_0\simeq X\times(\mathbb{P}^1\setminus\{0\})$ we view $(\Psi_i-\psi_D)_{\vert \mathcal{X}_{\tau}}$ as a function $\varphi_i^{\tau}\in \mathcal{C}^{\infty}(X)$. 
By Proposition \ref{second order variation multilinear energy} we then have
\begin{lem} Over $\mathbb{P}^1\setminus\{0\}$ we have $$
dd^c_{\tau} \langle \varphi_0^t, \dots, \varphi_n^t \rangle =\pi_*\left(\bigwedge_i (\Omega_i + dd^c\Psi_i)\right).
$$
\end{lem}
\begin{proof}
The result follows from Proposition \ref{second order variation multilinear energy} and the fact that $\mu$ is a biholomorphism away from $\tau = 0$, where also $\delta_D = 0$ (recalling that the $\mathbb{R}$-divisor $D$ is supported on $\mathcal{X}_0$).  
\end{proof}
\noindent Denoting by $u(\tau) := \langle \varphi_0^{\tau}, \dots, \varphi_n^{\tau} \rangle$ the Green-Riesz formula  then yields
$$
\frac{d}{dt}_{t=-\log\varepsilon} u(\tau)=\int_{\mathbb{P}^1\setminus\Delta_{\varepsilon}} dd^c_{\tau} u(\tau)=
$$
$$
\int_{\pi^{-1}(\mathbb{P}^1\setminus\Delta_{\varepsilon})}\bigwedge_i (\Omega_i + dd^c\Psi_i), 
$$
which converges to $(\mathcal{A}_0 \cdot \dots \cdot \mathcal{A}_n)$ as $\varepsilon\rightarrow 0$.  

It remains to show that 
\begin{equation*} 
 \lim_{t \rightarrow +\infty} \frac{ u(\tau)}{t} = \lim_{t \rightarrow +\infty} \frac{d}{dt} u(\tau),
\end{equation*} 
To see this, note that for each  closed $(1,1)$-form $\Theta$ on $\mathcal{X}$ and each smooth function $\Phi$ on $\mathcal{X}$, there is a K\"ahler form $\eta$ on $\mathcal{X}$ and a constant $C$ large enough so that $\Theta + C\eta + dd^c\Phi \geq 0$ on $\mathcal{X}$. Moreover, we have a relation
$$
\langle \varphi_0^t, \varphi_1^t, \dots, \varphi_n^t \rangle_{(\omega - \omega', \theta_1 \dots, \theta_n)}  = 
$$

$$
\langle {\varphi_0^t}, \varphi_1^t \dots, \varphi_n^t \rangle_{(\omega, \theta_1, \dots, \theta_n)} - \langle 0, \varphi_1^t \dots, \varphi_n^t \rangle_{(\omega', \theta_1, \dots, \theta_n)}
$$ 


\noindent and repeat this argument for each $i$, $0 \leq i \leq n$, by symmetry. It follows from the above 'multilinearity' that we can write $t \mapsto E(\varphi_0^t, \dots, \varphi_n^t)$ as a difference of convex functions,
concluding the proof. 
\end{proof}


\begin{rem} \label{Remark Theorem B}
The above proof in fact also yields a version of Theorem \ref{Theorem C} for \emph{subgeodesics} $(\varphi_i^t)_{t \geq 0}$ that are \emph{$L^{\infty}$-compatible} with smooth test configurations $(\mathcal{X}_i, \mathcal{A}_i)$ for $(X,\alpha_i)$ dominating $X \times \mathbb{P}^1$. This follows from the observation that one may replace $L^{\infty}$-compatible subgeodesic rays with smooth $\mathcal{C}^{\infty}$-compatible ones, using Lemma \ref{lem:smooth} and Lemma \ref{lemma: replacing ray}. 
\end{rem}

As a special case of Theorem B we obtain transcendental versions of several previously known formulas (see for instance \cite{BHJ2}). As an example, we may deduce the following formula for the asymptotics of the Monge-Amp\`ere energy functional by recalling that if $\omega$ is a K\"ahler form on $X$ and $(\varphi_t)_{t \geq 0}$ is a subgeodesic ray, then $$ \mathrm{E}(\varphi_t) = \frac{1}{(n+1)V} \langle \varphi_t, \dots, \varphi_t \rangle_{(\omega, \dots, \omega)}.$$

\begin{cor} \label{Corollary: Monge-Ampere asymptotic} Assume that $(\mathcal{X},\mathcal{A})$ is smooth and dominates $X\times\mathbb{P}^1$. For each smooth ray $(\varphi_t)_{t \geq 0}$ $\mathcal{C}^{\infty}$-compatible with $(\mathcal{X},\mathcal{A})$, we then have 
$$
\lim_{t\rightarrow +\infty}\frac{\mathrm{E}(\varphi_t)}{t}=\mathrm{E}^{\mathrm{NA}}(\mathcal{X},\mathcal{A})
$$
with 
$$
\mathrm{E}^{\mathrm{NA}}(\mathcal{X},\mathcal{A}):=\frac{(\mathcal{A}^{n+1})}{(n+1)V}.
$$
\end{cor}


\begin{rem} Here $\mathrm{E}^{\mathrm{NA}}$ makes reference to the \emph{non-Archimedean Monge-Amp\`ere energy functional}, see \cite{BHJ1} for an explanation of the terminology.
\end{rem}

To give a second example of an immediate corollary, interesting in its own right, we state the following (compare \cite{DervanRoss}):


\begin{cor} \label{cor J} Assume that $(\mathcal{X},\mathcal{A})$ is smooth and dominates $X\times\mathbb{P}^1$. For each smooth ray $(\varphi_t)_{t \geq 0}$ $\mathcal{C}^{\infty}$-compatible with $(\mathcal{X},\mathcal{A})$, we then have 
$$
\lim_{t\rightarrow +\infty}\frac{\mathrm{J}(\varphi_t)}{t}= \mathrm{J}^{\mathrm{NA}}(\mathcal{X}, \mathcal{A}),
$$
where $$\mathrm{J}^{\mathrm{NA}}(\mathcal{X},\mathcal{A}) := \frac{(\mathcal{A} \cdot \mu^*p_1^*\alpha^n)}{V} - \mathrm{E}^{\mathrm{NA}}(\mathcal{X},\mathcal{A}).$$ 
\end{cor}
\begin{proof}
Note that we may write $\mathrm{J}(\varphi_t) = V^{-1}\langle  \varphi_t, 0,\dots,0 \rangle_{(\omega, \dots, \omega)} - \mathrm{E}(\varphi_t)$ and apply Theorem \ref{Theorem C}.
\end{proof}


%
%

\smallskip

\section{Asymptotics for the K-energy} \label{proof A}
Let $(X,\omega)$ be a compact K\"ahler manifold and $\alpha := [\omega] \in H^{1,1}(X,\mathbb{R})$ a K\"ahler class on $X$. As before, let $(\mathcal{X}, \mathcal{A})$ be a smooth, relatively K\"ahler cohomological test configuration for $(X, \alpha)$ dominating $X \times \mathbb{P}^1$. In this section we explain how the above Theorem \ref{Theorem C} can be used to compute the asymptotic slope of the Mabuchi (K-energy) functional along rays $(\varphi^t)$, $\mathcal{C}^{1,1}_{\mathbb{C}}$-compatible with $(\mathcal{X},\mathcal{A})$. It is useful to keep the case of weak geodesic rays (as constructed in Lemma \ref{weak geodesic construction}) in mind, which in turn implies K-semistability of $(X,\alpha)$ (Theorem A). 

Regarding the proof of Theorem C, we will see that the Mabuchi functional is in fact of the form $\langle \varphi_0^t, \dots, \varphi_n^t \rangle_{(\theta_0, \dots, \theta_n)}$ for the appropriate choice of closed $(1,1)$-forms $\theta_i$ on $X$ and rays $(\varphi_i^t)$ on $X$, but Theorem \ref{Theorem C} does not directly apply in this situation. Indeed, the expression for the Mabuchi functional involves the metric $\log (\omega + dd^c\varphi_t)^n$ on $K_{\mathcal{X} / \mathbb{P}^1}$, which may blow up close to $\mathcal{X}_0$ (in particular, the compatibility conditions are not satisfied). However, a key point is that we can cook up a functional  $\mathrm{M}_{\mathcal{B}}$ of the above 'multivariate' formthat satisfies the same asymptotic slope as the Mabuchi functional (up to an explicit error term), and to which we may apply Theorem \ref{Theorem C}. 
More precisely, we show that
\begin{equation*}
\lim_{t \rightarrow +\infty} \frac{\mathrm{M}(\varphi_{t})}{t} = \lim_{t \rightarrow +\infty} \frac{\mathrm{M}_{\mathcal{B}}(\varphi_{t})}{t}  + V^{-1}((\mathcal{X}_{0,\mathrm{red}} - \mathcal{X}_0) \cdot \mathcal{A}^n)_{\mathcal{X}}, 
\end{equation*} 
and use Theorem \ref{Theorem C} to choose $\mathrm{M}_{\mathcal{B}}$ so that moreover $\lim_{t \rightarrow +\infty} \mathrm{M}_{\mathcal{B}}(\varphi_t)/t = \mathrm{DF}(\mathcal{X},\mathcal{A})$. It follows that the asymptotic slope of the Mabuchi (K-energy) functional equals $\mathrm{DF}(\mathcal{X}, \mathcal{A}) + V^{-1}((\mathcal{X}_{0,\mathrm{red}} - \mathcal{X}_0) \cdot \mathcal{A}^n)_{\mathcal{X}} =: \mathrm{M}^{\mathrm{NA}}(\mathcal{X}, \mathcal{A})$. 

\medskip

\subsection{A weak version of Theorem C} \label{Secion weak Theorem C}  We first explain how to obtain a weak version of Theorem C, as a direct consequence of Theorem \ref{Theorem C}. This version is more direct to establish than the full Theorem C, and will in fact be sufficient in order to prove K-semistability of $(X,\alpha)$, as explained in Section \ref{concluding the main theorem}. 

\begin{thm} \label{weak Theorem C}
Let $(\mathcal{X}, \mathcal{A})$ be a smooth, relatively K\"ahler cohomological test configuration for $(X,\alpha)$ dominating $X \times \mathbb{P}^1$. For each subgeodesic ray $(\varphi^t)_{t \geq 0}$, $\mathcal{C}^{1,1}_{\mathbb{C}}$-compatible with $(\mathcal{X},\mathcal{A})$, we have the inequality\footnote{The limit is in fact well-defined, as shown in Section \ref{Proof of strong Theorem C} below.} 
\begin{equation*}
\overline{\lim}_{t \rightarrow +\infty} \frac{\mathrm{M}(\varphi_{t})}{t} \leq \mathrm{DF}(\mathcal{X}, \mathcal{A}).
\end{equation*}  
\end{thm}

\begin{rem} In view of the strong version (see Theorem \ref{Reduced Theorem B}) we actually know that the limit is well-defined and, moreover, we obtain this way the precise asymptotic slope of the Mabuchi functional, see Section \ref{Proof of strong Theorem C}.
\end{rem} 

\begin{proof}[Proof of Theorem \ref{weak Theorem C}]
Let $\mathcal{B}$ be any smooth metric on $K_{ \mathcal{X} / \mathbb{P}^1} := K_{\mathcal{X}} - \pi^*K_{\mathbb{P}^1}$. Using the $\mathbb{C}^*$-action on $\mathcal{X}$ we can associate to $\mathcal{B}$ a ray of smooth metrics on $K_X$ that we denote by $(\beta_t)_{t \geq 0}$ (or $(\beta_{\tau})_{\tau \in \bar{\Delta}^*}$ for its reparametrisation by $t = -\log|\tau|$). Fix $\log \omega^n$ as a reference metric on $K_X$, and let
\begin{equation} \label{definition sigma}
\xi_{\mathcal{B}}^t :=  \log\left(\frac{e^{\beta_{\tau}}}{\omega^n} \right),
\end{equation}
i.e. the function given as the difference of metrics $\beta_{\tau} - \log \omega^n$ on $X$. The constructed ray $(\xi_{\mathcal{B}}^t)_{t \geq 0}$ is then $\mathcal{C}^{\infty}$-compatible with the cohomological test configuration $(\mathcal{X}, K_{\mathcal{X} / \mathbb{P}^1})$ for $(X,K_X)$. 

Now let $(\varphi_t)_{t \geq 0}$ be any subgeodesic ray $\mathcal{C}^{1,1}_{\mathbb{C}}$-compatible with $(\mathcal{X}, \mathcal{A})$. By Lemma \ref{lem:smooth}, Lemma \ref{lemma: replacing ray} and Theorem \ref{Theorem C} it follows that 

\begin{equation} \label{canonical limit}
\frac{ \langle \xi_{\mathcal{B}}^t, \varphi_t \dots, \varphi_t \rangle_{(-\ric \omega, \omega \dots, \omega)}}{t} \longrightarrow (K_{\mathcal{X}/\mathbb{P}^1} \cdot \mathcal{A}^n)_{\mathcal{X}}
\end{equation}

\noindent as $t \rightarrow +\infty$. Indeed, by Lemma \ref{lem:smooth} we may choose a smooth subgeodesic ray $(\varphi'_t)_{t\geq 0}$ in $\mathcal{H}$ that is $\mathcal{C}^{\infty}$-compatible (and hence also $L^{\infty}$- and $\mathcal{C}^{1,1}_{\mathbb{C}}$-compatible) with $(\mathcal{X},\mathcal{A})$. Up to replacing $(\varphi_t)$ with $(\varphi'_t)$ we may thus assume that $(\varphi_t)$ is smooth and $\mathcal{C}^{\infty}$-compatible with $(\mathcal{X},\mathcal{A})$, using Lemma \ref{lemma: replacing ray}, so that Theorem \ref{Theorem C} applies. 


\noindent Motivated by the Chen-Tian formula \eqref{Chens formula} and the identity \eqref{canonical limit}, we thus introduce the notation 
$$
\mathrm{M}_{\mathcal{B}}(\varphi_t) := \bar{\mathcal{S}} \mathrm{E}(\varphi_{\tau}) + V^{-1} \langle \xi_{\mathcal{B}}^t, \varphi_t \dots, \varphi_t \rangle_{(-\ric \omega, \omega \dots, \omega)},
$$ 
the point being that the asymptotic slope of this functional coincides with the Donaldson-Futaki invariant (even when the central fiber is not reduced). 

\begin{lem} \label{F converges to DF equation} $$\lim_{t \rightarrow +\infty} \frac{\mathrm{M}_{\mathcal{B}}(\varphi_t)}{t} = \mathrm{DF}(\mathcal{X}, \mathcal{A})$$.
\end{lem}
\begin{proof}
This result is an immediate consequence of \eqref{canonical limit}, the Chen-Tian formula \eqref{Chens formula} and Corollary \ref{Corollary: Monge-Ampere asymptotic}.
\end{proof}

\noindent Hence, it suffices to establish the following inequality
$$
\overline{\lim}_{t \rightarrow +\infty} \frac{\mathrm{M}(\varphi_t)}{t} \leq \lim_{t \rightarrow +\infty} \frac{\mathrm{M}_{\mathcal{B}}(\varphi_t)}{t}.
$$
\noindent To do this, we set
$
\Gamma(\tau) := (\mathrm{M} - \mathrm{M}_{\mathcal{B}})(\varphi_t).
$
By the Chen-Tian formula \eqref{Chens formula} and cancellation of terms we have
\begin{equation*}
\Gamma(\tau) = \bar{\mathcal{S}} \mathrm{E}(\varphi_t)  - \mathrm{E}^{\ric \omega}(\varphi_t) + V^{-1} \int_X \log \left( \frac{(\omega + dd^c\varphi_{\tau})^n}{\omega^n} \right) (\omega + dd^c\varphi_{\tau})^n
\end{equation*}
 
$$
- \bar{\mathcal{S}} \mathrm{E}(\varphi_t) - V^{-1} \langle \xi_{\mathcal{B}}^t, \varphi_t \dots, \varphi_t \rangle_{(-\ric \omega, \omega \dots, \omega)}
$$

$$
= - \mathrm{E}^{\ric \omega}(\varphi_t) + V^{-1} \int_X \log \left( \frac{(\omega + dd^c\varphi_{\tau})^n}{\omega^n} \right) (\omega + dd^c\varphi_{\tau})^n - V^{-1} \int_X  \xi_{\mathcal{B}}^t \; (\omega + dd^c\varphi_{\tau})^n 
$$

$$
+ V^{-1} \sum_{j = 0}^{n-1} \int_X \varphi_t \; \ric{\omega} \wedge \omega^j \wedge (\omega + dd^c\varphi_t)^{n-j-1} 
$$

$$
= V^{-1} \int_X \log \left( \frac{(\omega + dd^c\varphi_{\tau})^n}{\omega^n} \right) (\omega + dd^c\varphi_{\tau})^n - V^{-1} \int_X  \log\left(\frac{e^{\beta_{\tau}}}{\omega^n} \right) \; (\omega + dd^c\varphi_{\tau})^n 
$$

$$
= V^{-1} \int_X \log \left( \frac{(\omega + dd^c\varphi_{\tau})^n}{e^{\beta_{\tau}}} \right) (\omega + dd^c\varphi_{\tau})^n,
$$

\noindent recalling the definition \eqref{definition sigma} of $\xi_{\mathcal{B}}^t$ and Definition \ref{Multivariate energy def}. 


In view of Proposition \ref{divisor representation}, we as usual let $D$ denote the unique $\mathbb{R}$-divisor supported on $\mathcal{X}_0$ such that 
$
\mathcal{A}=\mu^*p_1^*\alpha+[D],
$
with $p_1:X \times \mathbb{P}^1 \rightarrow X$ the first projection.
Fix a choice of an $S^1$-invariant function 'Green function' $\psi_D$ for $D$, so that $\delta_D=\theta_D+dd^c\psi_D$ with $\theta_D$ a smooth $S^1$-invariant closed $(1,1)$-form on $\mathcal{X}$. Moreover, set $\Omega := \mu^*p_1^*\alpha + \theta_D$ (for which $[\Omega] = \mathcal{A}$ then holds) and let $\Phi$ denote the $S^1$-invariant function on $X \times \mathbb{P}^1$ corresponding to the ray $(\varphi_t)$. In particular, the function $\Phi \circ \mu + \psi_D$ extends to a smooth $\Omega$-psh function $\Psi$ on $\mathcal{X}$, by $\mathcal{C}^{\infty}$-compatibility.

With the above notation in place, the integrand in the above expression for $\Gamma(\tau)$ can be written 

\begin{equation*} \label{integrand}
\log \left( \frac{(\omega + dd^c\varphi_{\tau})^n}{e^{\beta_{\tau}}} \right) = \mu_* \left( \log \left( \frac{(\Omega + dd^c\Psi)^n \wedge \pi^*(\sqrt{-1} \; d\tau \wedge d\bar{\tau})}{\lambda_{\mathcal{B}}} \right) \right),
\end{equation*}
where 
\begin{equation*} 
\lambda_{\mathcal{B}} := e^{\mathcal{B} +  \pi^*\log(\sqrt{-1} \; d\tau \wedge d\bar{\tau})}
\end{equation*}
is the volume form defined by the smooth metric $\mathcal{B} + \pi^*\log(\sqrt{-1} \; d\tau \wedge d\bar{\tau})$ on $K_{\mathcal{X}}$. Since $\Psi$ is $\Omega$-psh on $\mathcal{X}$ and $\lambda_{\mathcal{B}}$ is a volume form on $\mathcal{X}$, this quantity is bounded from above. Moreover, we integrate against the measure $(\omega + dd^c\varphi_{\tau})^n$ which can be computed in cohomology, thus has mass independent of $\tau$. Hence
\begin{equation*}
\Gamma(\tau) = V^{-1} \int_X \log \left( \frac{(\omega + dd^c\varphi_{\tau})^n}{e^{\beta_{\tau}}} \right) (\omega + dd^c\varphi_{\tau})^n \leq O(1).
\end{equation*}
Dividing by $t$ and passing to the limit now concludes the proof. 
\end{proof}

\noindent As explained below, the above 'weak Theorem C' actually suffices to yield our main result. 


\subsection{Proof of Theorem A} \label{concluding the main theorem}
We now explain how the above considerations apply to give a proof of Theorem A and point out some immediate and important consequences regarding the YTD conjecture. 
First recall the following definition (see e.g. \cite[Section 7.2]{Tianlecturenotes}): 

\begin{mydef} \label{definition uniform K-stability} We say that $(X, \alpha)$ is \emph{uniformly K-stable} if there is a $\delta > 0$ and $C \geq 0$ such that 
$$
\mathrm{M}^{\mathrm{NA}}(\mathcal{X}, \mathcal{A}) \geq \delta \mathrm{J}^{\mathrm{NA}}(\mathcal{X}, \mathcal{A}) - C
$$
for all relatively K\"ahler cohomological test configurations $(\mathcal{X}, \mathcal{A})$ for $(X,\alpha)$. 
\end{mydef}

\noindent We are now ready to prove Theorem A.

\begin{proof}[Proof of Theorem A]

Let $X$ be a compact K\"ahler manifold and $\omega$ a given K\"ahler form, with $\alpha := [\omega]  \in H^{1,1}(X,\mathbb{R})$ the corresponding K\"ahler class. Let $(\mathcal{X},\mathcal{A})$ be any (possibly singular) cohomological test configuration for $(X,\alpha)$ which by desingularisation and perturbation (see Proposition \ref{lemma good divisor}) can be assumed to be smooth, relatively K\"ahler and dominating $X \times \mathbb{P}^1$. 
Consider any ray $(\varphi_t)_{t \geq 0}$ such that Theorem C applies; for instance one may take $(\varphi_t)$ to be the associated weak geodesic ray emanating from $\omega$ (i.e. such that $\varphi_0 = 0$), which due to \cite{Chen00} (cf. also \cite{Blocki13}, \cite{Darvas14},  \cite{DL12}) is $\mathcal{C}^{1,1}_{\mathbb{C}}$-compatible with $(\mathcal{X}, \mathcal{A})$.
Now suppose that the Mabuchi functional is bounded from below (in the given class $\alpha$). In particular, we then have
$$
\mathrm{DF}(\mathcal{X},\mathcal{A}) \geq \overline{\lim}_{t \rightarrow +\infty} \frac{\mathrm{M}(\varphi_t)}{t} \geq 0,
$$
using the weak version of Theorem C, cf. Theorem \ref{weak Theorem C}. Since the cohomological test configuration $(\mathcal{X}, \mathcal{A})$ for $(X,\alpha)$ was chosen arbitrarily, this proves Corollary \ref{cor main}, i.e. shows that $(X,\alpha)$ is K-semistable. 

In a similar vein, suppose that the Mabuchi functional is coercive, i.e. in particular $
\mathrm{M}(\varphi_t) \geq \delta \mathrm{J}(\varphi_t) - C
$ for some constants $\delta, C > 0$ uniform in $t$. 
Note that Corollary \ref{cor J}) and the (weak) Theorem C provides a link with the intersection theoretic quantities $\mathrm{J}^{\mathrm{NA}}(\mathcal{X}, \mathcal{A})$  and $\mathrm{M}^{\mathrm{NA}}(\mathcal{X}, \mathcal{A})$ respectively. More precisely, dividing by $t$ and passing to the limit we have
$$
0 \leq \overline{\lim}_{t \rightarrow +\infty} \frac{(\mathrm{M} - \delta \mathrm{J})(\varphi_t)}{t} \leq \mathrm{M}^{\mathrm{NA}}(\mathcal{X}, \mathcal{A}) - \delta \mathrm{J}^{\mathrm{NA}}(\mathcal{X}, \mathcal{A}).
$$
Since $(\mathcal{X}, \mathcal{A})$ was chosen arbitrarily it follows that $(X,\alpha)$ is uniformly K-stable, concluding the proof of Theorem A. 
\end{proof}





As remarked in the introduction it follows from convexity of the Mabuchi functional along weak geodesic rays, cf. \cite{BB, ChenLiPaun}, that the Mabuchi functional is bounded from below (in the given class $\alpha$) if $\alpha$ contains a cscK representative. In other words, Corollary \ref{cor main} follows. 

Moreover, it is shown in \cite[Theorem 1.2]{BDL} that the Mabuchi functional $\mathrm{M}$ is in fact coercive if $\alpha$ contains a cscK representative. As a consequence, we obtain also the following stronger result, confirming the "if" direction of the YTD conjecture (here referring to its natural generalisation to the transcendental setting, using the notions introduced in Section \ref{test config}). 

\begin{cor}
If the K\"ahler class $\alpha \in H^{1,1}(X,\mathbb{R})$ admits a constant scalar curvature representative, then $(X,\alpha)$ is uniformly K-stable. 
\end{cor}

\smallskip

\subsection{Asymptotic slope of the K-energy} \label{Proof of strong Theorem C}
Building on Section \ref{Secion weak Theorem C} we now improve on the weak version of Theorem C (cf. Theorem \ref{weak Theorem C}) by computing the asymptotic slope of the Mabuchi (K-energy) functional (even when the central fiber is not reduced). To this end, recall the definition of the non-Archimedean Mabuchi functional, i.e. the intersection number 
$$
\mathrm{M}^{\mathrm{NA}}(\mathcal{X}, \mathcal{A}) :=  \mathrm{DF}(\mathcal{X}, \mathcal{A}) + V^{-1}((\mathcal{X}_{0,\mathrm{red}} - \mathcal{X}_0) \cdot \mathcal{A}^n)_{\mathcal{X}},
$$ 
discussed in Section \ref{further reduction techniques}. Note that it satisfies $\mathrm{M}^{\mathrm{NA}}(\mathcal{X}, \mathcal{A}) \leq  \mathrm{DF}(\mathcal{X}, \mathcal{A})$ with equality precisely when the central fiber is reduced. 


\noindent Adapting the techniques of \cite{BHJ2} to the present setting we now obtain the following result, corresponding to Theorem C of the introduction. 

\begin{thm} \label{Reduced Theorem B} 
Let $X$ be a compact K\"ahler manifold and $\alpha \in H^{1,1}(X,\mathbb{R})$ a K\"ahler class. Suppose that $(\mathcal{X}, \mathcal{A})$ is a smooth, relatively K\"ahler cohomological test configuration for $(X,\alpha)$ dominating $X \times \mathbb{P}^1$. Then, for each subgeodesic ray $(\varphi_t)_{t \geq 0}$, $\mathcal{C}^{1,1}_{\mathbb{C}}$-compatible with $(\mathcal{X},\mathcal{A})$, the asymptotic slope of the Mabuchi functional is well-defined and satisfies
\begin{equation*}
\frac{\mathrm{M}(\varphi_{t})}{t} \longrightarrow \mathrm{M}^{\mathrm{NA}}(\mathcal{X}, \mathcal{A})
\end{equation*}
as $t \rightarrow +\infty$.  
\end{thm}

\begin{rem} In particular, this result holds when $(\varphi_t)_{t \geq 0}$ is the weak geodesic ray associated to $(\mathcal{X},\mathcal{A})$, constructed in Section \ref{smoothcompatibility}.
\end{rem}

\begin{proof}[Proof of Theorem \ref{Reduced Theorem B}]
Following ideas of \cite{BHJ2} we associate to the given smooth, relatively K\"ahler and dominating test configuration $(\mathcal{X}, \mathcal{A})$ for $(X,\alpha)$ another test configuration $(\mathcal{X}', \mathcal{A}')$ for $(X,\alpha)$ which is semistable, i.e. smooth and such that $\mathcal{X}'_0$ is a reduced $\mathbb{R}$-divisor with simple normal crossings. As previously noted, we can also assume that $\mathcal{X}'$ dominates the product. In the terminology of Section \ref{further reduction techniques}, this construction comes with a morphism $ g_d \circ \rho: \mathcal{X}' \rightarrow \mathcal{X}$, cf. the diagram in Section \ref{further reduction techniques}. Pulling back, we set $\mathcal{A}' := g_d^*\rho^*\mathcal{A}$. Note that $\mathcal{A}'$ is no longer relatively K\"ahler, but merely relatively semipositive (with the loss of positivity occuring along $\mathcal{X}_0'$).  

On the one hand, Lemma \ref{homogeneity of mabuchi} yields
\begin{equation} \label{equation one}
\mathrm{M}^{\mathrm{NA}}(\mathcal{X}', \mathcal{A}')  = d \cdot \mathrm{M}^{\mathrm{NA}}(\mathcal{X}, \mathcal{A}),
\end{equation}
where $d > 0$ is the degree of the morphism $g_d$. On the other hand, we may consider the pull back by $g_d \circ \rho$ of the weak geodesic $(\varphi_t)_{t \geq 0}$ associated to $(\mathcal{X}, \mathcal{A})$.  This induces a subgeodesic $({\varphi}_t')_{t \geq 0}$ which is  $\mathcal{C}^{1,1}_{\mathbb{C}}$-compatible with the test configuration $(\mathcal{X}', \mathcal{A}')$ for $(X, \alpha)$ (in particular, the boundedness of the Laplacian is preserved under pullback by $g_d \circ \rho$). 
Replacing $\tau$ by $\tau^d$ amounts to replacing $t$ by $d \cdot t$, so that
\begin{equation}\label{equation two}
\frac{\mathrm{M}({\varphi}_t')}{t} = d \cdot \frac{\mathrm{M}(\varphi_t)}{t}.
\end{equation}
Combining equations \eqref{equation one} and \eqref{equation two} it thus follows that $$\lim_{t \rightarrow +\infty}\frac{\mathrm{M}(\varphi_{t})}{t} = \mathrm{M}^{\mathrm{NA}}(\mathcal{X}, \mathcal{A})$$
if and only if 
\begin{equation} \label{equation three}
\lim_{t \rightarrow +\infty}\frac{\mathrm{M}({\varphi}_{t}')}{t} = \mathrm{DF}(\mathcal{X}',\mathcal{A}').
\end{equation}

\noindent In other words, it suffices to establish \eqref{equation three} above. By the asymptotic formula \ref{F converges to DF equation} it is in turn equivalent to show that
\begin{equation*}
\lim_{t \rightarrow +\infty} \frac{\mathrm{M}({\varphi}_{t}')}{t} = \lim_{t \rightarrow +\infty} \frac{\mathrm{M}_{\mathcal{B}}({\varphi}_{t}')}{t}.
\end{equation*}
We use the notation of the proof of Theorem \ref{weak Theorem C}. In particular, we set $\Gamma(\tau) := (\mathrm{M} - \mathrm{M}_{\mathcal{B}})(\varphi_{\tau}')$. 
As in the proof of Theorem \ref{weak Theorem C} we have an upper bound $\Gamma(\tau) \leq O(1)$, using that the restriction of the relatively semipositive class $\mathcal{A}'$ to $\mathcal{X}'\setminus \mathcal{X}_0'$ is in fact relatively K\"ahler. 

To obtain a lower estimate of $\Gamma(\tau)$ we consider the Monge-Ampere measure $\mathrm{MA}(\varphi_{\tau}') := V^{-1} (\omega + dd^c\varphi_{\tau}')^n$ and note that
\begin{equation*}
V^{-1} \Gamma(\tau) = V^{-1} \int_X \log \left( \frac{(\omega + dd^c\varphi_{\tau}')^n}{e^{\beta_{\tau}}} \right) = 
\end{equation*}

\begin{equation*}
= \int_X \log \left( \frac{\mathrm{MA}(\varphi_{\tau}')}{e^{\beta_{\tau}} / \int_X e^{\beta_{\tau}}} \right) \mathrm{MA}(\varphi_{\tau}') - \log \int_X e^{\beta_{\tau}} \geq - \log \int_X e^{\beta_{\tau}},
\end{equation*}
since the relative entropy of the two probability measures $\mathrm{MA}(\varphi_{\tau})$ and $e^{\beta_{\tau}} / \int_X e^{\beta_{\tau}}$ is non-negative. 
We now conclude by estimating this integral, using the following result from \cite{BHJ2}: 

\begin{lem} \label{BHJ estimate} \emph{(\cite{BHJ2})}. Let $(\mathcal{X}, \mathcal{A})$ be a semistable and dominating test configuration for $(X,\alpha)$ and let $\mathcal{B}$ be any smooth metric on $K_{ \mathcal{X}/ \mathbb{P}^1}$. Let $(\beta^t)_{t \geq 0}$ be the family of smooth metrics on $K_X$ induced by $\mathcal{B}$. Denote by $p \geq 1$ the largest integer such that $p-1$ distinct irreducible components of $\mathcal{X}_0$ have a non-empty intersection. Then there are positive constants $A$ and $B$ such that 
\begin{equation*}
A t^{2(p-1)} \leq \int_X e^{\beta^t} \leq B t^{2(p-1)}.
\end{equation*}
holds for all $t$.
\end{lem}

We refer the reader to \cite{BHJ2} for the proof and here simply apply the result: Recalling that $t = -\log|\tau|$, Lemma \ref{BHJ estimate} yields that $\log\int_X e^{\beta_{\tau}} = o(t)$ and so it follows from 
\begin{equation*}
\frac{-\log \int_X e^{\beta_{\tau}}}{t} \leq \frac{\Gamma(\tau)}{t} \leq \frac{O(1)}{t}
\end{equation*}
that 
\begin{equation*}
\lim_{t \rightarrow +\infty} \frac{\Gamma(\tau)}{t} = 0,
\end{equation*}
completing the proof.
\end{proof}

\bibliography{ksemistability} 
\bibliographystyle{amsalpha}

\end{document}